\documentclass[final,3p,times]{elsarticle} 
\usepackage[T1]{fontenc}

\usepackage{amssymb}

\usepackage{amsthm}

\journal{INS}

\usepackage[english]{babel}
\usepackage{srcltx}
\usepackage{color,amsthm,hyperref}

\newcommand*{\normally}{\mathrel{\ooalign{$|$\hfil\cr\kern+1pt$\thicksim$}}} 
\def\True{\mbox{True}}
\def\False{\mbox{False}}
\def\Void{\mbox{Void}}

\def\myscale{0.7} 
\def\myscale{0.52} 
\def\s~{| \hspace{-1.8 mm} \sim}
\def\C{{\cal C}}
\def\D{{\cal D}}
\def\F{{\cal F}}

\def\I{{\cal I}}
\def\H{{\cal H}}
\def\K{{\cal K}}
\def\L{{\cal L}}
\def\P{{\cal P}}
\def\S{{\cal S}}
\def\U{{\cal U}}

\theoremstyle{definition} 
\newtheorem{definition}{Definition}
\newtheorem{proposition}{Proposition}
\newtheorem{theorem}{Theorem}
\newtheorem{corollary}{Corollary}
\newtheorem{remark}{Remark}

\newtheorem{example}{Example}

\begin{document}

\begin{frontmatter}
\title{Quasi Conjunction, Quasi Disjunction, T-norms and T-conorms: \\ Probabilistic Aspects }
\tnotetext[label0]{Both authors contributed equally to this work}
\author[ag]{Angelo Gilio}
\ead{angelo.gilio@sbai.uniroma1.it}
\author[gs]{Giuseppe Sanfilippo\corref{cor1}}
\ead{giuseppe.sanfilippo@unipa.it}
\cortext[cor1]{Corresponding author}
\address[ag]{Dipartimento di Scienze di Base e Applicate per l'Ingegneria,
University of Rome  ``La Sapienza'',
Via Antonio Scarpa 16, 00161 Roma, Italy}
\address[gs]{Dipartimento di Matematica e Informatica, University of Palermo,
Via Archirafi 34, 90123 Palermo, Italy
}
\begin{abstract}
We make a probabilistic analysis related to some inference rules
which play an important role in nonmonotonic reasoning. In a
coherence-based setting, we study the extensions of a probability
assessment defined on $n$ conditional events to their quasi
conjunction, and by exploiting duality, to their quasi disjunction.
The lower and upper bounds coincide with some well known t-norms and
t-conorms:  minimum, product, Lukasiewicz,  and Hamacher t-norms
and their dual t-conorms.  On this basis we obtain Quasi And and
Quasi Or rules.  These are rules for which any finite family of
conditional events p-entails the associated quasi conjunction and
quasi disjunction. We examine some cases of logical dependencies,
and we study the relations among coherence, inclusion for
conditional events, and p-entailment. We also consider the Or
rule, where quasi conjunction and quasi disjunction of premises
coincide with the conclusion. We analyze further aspects of quasi
conjunction and quasi disjunction, by computing probabilistic bounds
on premises from bounds on conclusions. Finally, we consider
biconditional events, and we introduce the notion of an
$n$-conditional event.  Then we give a probabilistic interpretation
for a generalized Loop rule. In an appendix we provide explicit
expressions for the Hamacher t-norm and t-conorm in the unitary
hypercube.
\end{abstract}
\end{frontmatter}
\section{Introduction}
In classical (monotonic) logic, if a conclusion $C$ follows from
some premises, then $C$ also follows when the set of premises is
enlarged; that is, adding premises never invalidates any
conclusions. In contrast, in (nonmonotonic) commonsense reasoning we
are typically in a situation of partial knowledge, and a conclusion
reached from a set of premises may be retracted when some premises
are added. Nonmonotonic reasoning is a relevant topic in the field
of artificial intelligence, and has been studied in literature by
many symbolic and numerical formalisms (see, e.g.
\cite{BeDP97,BGLS02,BGLS05,DuPr94,KrLM90}). A remarkable theory
related to nonmonotonic reasoning has been proposed by Adams in his
probabilistic logic of conditionals (\cite{Adam75}). We recall that
the approach of Adams can be developed with full generality by
exploiting  coherence-based probabilistic reasoning (\cite{deFi70}).
In the setting of coherence  conditional probabilities can be  directly assigned, and  {\em zero probabilities for conditioning events}  can be properly managed
(see, e.g.
\cite{BiGi00,BiGS03,BiGS03b,BiGS12,CoSc02,Gili02,GiSa11a,GoMa06,Sanf12}).
The coherence-based approach is applied in many fields: statistical
analysis,  decision theory, probabilistic default reasoning and
fuzzy theory.  It allows one to manage incomplete probabilistic
assignments in a situation of vague or partial knowledge (see, e.g.
\cite{BrCV12,CaLS07,CaRV10,CGTV12,CoSV12,Gili12,Lad96,LaSA12,ScVa09}).
A basic notion in the work of Adams is the quasi conjunction of conditionals.  This logical operation also plays a relevant role in \cite{DuPr94} (see also \cite{BeDP97}), where a suitable Quasi And rule is introduced to characterize entailment from a knowledge base. In the present article, besides quasi conjunction, we study by duality  the quasi disjunction of conditional events and the associated Quasi Or rule. \\
Theoretical tools which play a relevant role in artificial intelligence and fuzzy logic are t-norms and t-conorms.  These allow one to extend the Boolean operations of conjunction and disjunction to the setting of multi-valued logics. T-norms (first proposed in \cite{Meng42}) and t-conorms were introduced in \cite{ScSk61} and are a subclass of aggregation functions (\cite{GMMP09,GMMP11b,GMMP11}).  They play a basic role in decision theory, information and data fusion, probability theory and risk management.\\
In this paper we give many insights  about probabilistic default
reasoning in the setting of coherence, by making a probabilistic
analysis of the  Quasi And, Quasi Or and Loop inference rules.  Some
results  were already given without proof in \cite{GiSa10}. To
begin, we recall some basic notions and results regarding coherence,
probabilistic default reasoning, and the Hamacher t-norm/t-conorm
(Section 2). Then, we show that some well known t-norms and
t-conorms appear as lower and upper bounds when we propagate
probability assessments on a finite family of conditional events to
the associated quasi conjunction.  By these bounds we obtain the
Quasi And rule.  We also consider  special cases of logical
dependencies associated with the Goodman-Nguyen inclusion relation
(\cite{GoNg88}) and with the compound probability theorem.  Then, we
give two results which identify the strict relationship holding
among coherence, the Goodman-Nguyen inclusion relation, and
p-entailment (Section 3). We deepen a further aspect
of the Quasi And rule by determining the probability bounds on the
premises from given bounds on the conclusion of the rule (Section 4). By
exploiting duality, we give  analogous results  for the quasi
disjunction of conditional events, and we  obtain the Quasi Or rule.
We also examine the Or rule, and we show that quasi conjunction and
quasi disjunction of the premises of this rule both coincide with
its conclusion (Section 5). In a similar way, we then enrich the
Quasi Or rule by determining the probability bounds on the premises
from given bounds on the conclusion of the rule (Section 6).
We consider biconditional events, and we introduce the
notion of an $n$-conditional event, by means of which we give a
probabilistic semantics to a \emph{generalized Loop rule} (Section 7).  Finally,
we give some conclusions and perspectives on future work (Section
8). We illustrate  notions and results with a table and some
figures.\\
The results given in this work may be useful for the treatment of
uncertainty in many applications of statistics and artificial
intelligence, in particular for the probabilistic approach to
inference rules in nonmonotonic reasoning, for the psychology of
uncertain reasoning, and for probabilistic reasoning in the semantic
web (see, e.g., \cite{GiOv12,Klei13,LuSt08,PfKl06,PfKl09}).

\section{Some Preliminary Notions} \label{PRELIM}
In this section we first discuss some basic notions regarding
coherence.  Then, we recall the notions of p-consistency and
p-entailment of Adams (\cite{Adam75}) within the setting of
coherence.
\subsection{Basic notions  on coherence}
As in the approach of de Finetti, events represent uncertain facts
described by  (non ambiguous) logical propositions.  An event $A$ is
a two-valued logical entity which can be true ($T$), or false ($F$).
The indicator of $A$ is a two-valued numerical quantity which is 1,
or 0, according to  whether $A$ is true, or false. We denote by
$\Omega$ the sure event and by $\emptyset$ the impossible one. We
use the same symbols for events and their indicators. Moreover, we
denote by $A\land B$ (resp., $A \vee B$) the logical intersection,
or conjunction (resp., logical union, or disjunction). To simplify
notations, in many cases we  denote the conjunction between $A$ and
$B$ as the product $AB$. We denote by $A^c$ the negation of $A$. Of
course, the truth values for conjunctions, disjunctions and
negations are obtained by applying the propositional calculus. Given
any events $A$ and $B$, we simply write $A \subseteq B$ to denote
that $A$ logically implies $B$, that is  $AB^c=\emptyset$, which
means that $A$ and $B^c$ cannot be both true. Given  $n$ events
$A_1, \ldots, A_n$, as $A_i \vee A_i^c = \Omega \,,\;\; i = 1,
\ldots, n$,  by expanding the expression $\bigwedge_{i=1}^n(A_i \vee
A_i^c)$, we can represent $\Omega$ as the disjunction of $2^n$
logical conjunctions, some of which may be impossible.  The
remaining ones are the atoms, or constituents, generated by $A_1,
\ldots, A_n$. We recall that $A_1, \ldots, A_n$ are logically
independent when the number of atoms generated by them is $2^n$. Of
course, in case of some logical dependencies among $A_1, \ldots,
A_n$ the number of atoms is less than $2^n$. For instance, given two
logically incompatible events $A,B$, as $AB=\emptyset$ the atoms
are: $AB^c, A^cB, A^cB^c$. We remark that, to introduce the basic
notions, an equivalent approach is that of considering a Boolean
algebra $\mathcal{B}$ whose elements are interpreted as events. In
this way events would be combined by means of the Boolean
operations; then to say that $A_1, \ldots, A_n$  are logically
independent would mean that the subalgebra generated by them has
$2^n$ atoms. Concerning conditional events, given two events $A,B$,
with $A \neq \emptyset$, in our approach the conditional event $B|A$
is defined as a three-valued logical entity which is true (T), or
false (F), or void (V), according to whether $AB$ is true, or $AB^c$
is true, or $A^c$ is true, respectively. We recall that, agreeing to
the betting metaphor, if you assess $P(B|A)=p$, then you are willing
to pay an amount $p$ and  to receive 1, or 0, or $p$, according to
whether $AB$ is true, or $AB^c$ is true, or $A^c$ is true (bet
called off), respectively. Given a real function $P : \; \mathcal{F}
\, \rightarrow \, \mathcal{R}$, where $\mathcal{F}$ is an arbitrary
family of conditional events, let us consider a subfamily
$\mathcal{F}_n = \{E_1|H_1, \ldots, E_n|H_n\} \subseteq
\mathcal{F}$, and the vector $\mathcal{P}_n =(p_1, \ldots, p_n)$,
where $p_i = P(E_i|H_i) \, ,\;\; i = 1, \ldots, n$. We denote by
$\mathcal{H}_n$ the disjunction $H_1 \vee \cdots \vee H_n$. As
$E_iH_i \vee E_i^cH_i \vee H_i^c = \Omega \,,\;\; i = 1, \ldots, n$,
by expanding the expression $\bigwedge_{i=1}^n(E_iH_i \vee E_i^cH_i
\vee H_i^c)$, we can represent $\Omega$ as the disjunction of $3^n$
logical conjunctions, some of which may be impossible.  The
remaining ones are the atoms, or constituents, generated by the
family $\mathcal{F}_n$ and, of course, are a partition of $\Omega$.
We denote by $C_1, \ldots, C_m$ the constituents contained in
$\mathcal{H}_n$ and (if $\mathcal{H}_n \neq \Omega$) by $C_0$ the
remaining constituent $\mathcal{H}_n^c = H_1^c \cdots H_n^c $, so
that
\[
\mathcal{H}_n = C_1 \vee \cdots \vee C_m \,,\;\;\; \Omega =
\mathcal{H}_n^c \vee
\mathcal{H}_n = C_0 \vee C_1 \vee \cdots \vee C_m \,,\;\;\; m+1 \leq 3^n
\,.
\]
{\em Interpretation with the betting scheme}. With the pair $(\mathcal{F}_n, \mathcal{P}_n$) we associate the random gain
${\mathcal{G}} = \sum_{i=1}^n s_iH_i(E_i - p_i)$,
where $s_1, \ldots, s_n$ are $n$ arbitrary real numbers. We observe that ${\mathcal{G}}$ is the difference between the amount that you receive, $\sum_{i=1}^n s_i(E_iH_i + p_iH_i^c)$, and the amount that you pay, $\sum_{i=1}^n s_ip_i$, and represents the net gain from engaging each transaction $H_i(E_i - p_i)$, the scaling and meaning (buy or sell) of the transaction being specified by the magnitude and the sign of $s_i$ respectively. Let $g_h$
be the value of $\mathcal{G}$ when $C_h$ is true; of course, $g_0 = 0$.
Denoting by $G_{\mathcal{H}_n}=\{g_1, \ldots, g_m\}$ the set of values of $\mathcal{G}$ restricted to $\H_n$, we have
\begin{definition}\label{COER-BET} {\rm The function $P$ defined on $\mathcal{F}$ is said to be {\em coherent}
if and only if, for every integer $n$, for every finite sub-family $\mathcal{F}_n$
$\subseteq \mathcal{F}$ and for every $s_1, \ldots, s_n$, one has:
$\min  G_{\mathcal{H}_n} \leq 0 \leq \max G_{\mathcal{H}_n}$.}
\end{definition}
Notice that the condition $\min G_{\mathcal{H}_n}  \leq 0 \leq \max G_{\mathcal{H}_n}$ can be written in two equivalent ways: $\min  G_{\mathcal{H}_n} \leq 0$, or  $\max  G_{\mathcal{H}_n} \geq 0$.  As shown by Definition \ref{COER-BET}, a probability assessment is coherent if and only if, in any finite combination of $n$ bets, it does not happen that the values $g_1, \ldots, g_m$ are all positive, or all negative ({\em no Dutch Book}). \\ \ \\
{\em Coherence with penalty criterion.}
An equivalent notion of coherence for unconditional events and random quantities was introduced by de Finetti (\cite{deFi62,deFi64,deFi70}) using the penalty criterion associated with the quadratic scoring rule. Such a  penalty criterion has been extended to the case of conditional events in \cite{Gili90}. With the pair $(\F_n,\P_n$) we associate the loss $ \mathcal{L} = \sum_{i=1}^nH_i(E_i - p_i)^2$; we denote by $L_h$ the value of $\mathcal{L}$ if $C_h$ is true. If you specify the assessment $\P_n$ on $\F_n$ as representing your belief's degrees, you are required to pay a penalty $L_h$ when $C_h$ is
true. Then, we have
\begin{definition}\label{COER-PENALTY}{\rm
The function $P$ defined on $\F$ is said to be {\em coherent} if and
only if there does not exist an integer $n$, a finite sub-family
$\F_n$ $ \subseteq \mathcal{\F}$, and an assessment ${\P_n}^*$ $ =
(p_1^*, \ldots, p_n^*)$ on $\F_n$ such that, for the loss $
\mathcal{L}^* = \sum_{i=1}^nH_i(E_i - p_i^*)^2,$ associated with
$(\F_n,\P_n^*)$, it results $\mathcal{L}^* \leq \mathcal{L}$ and
$\mathcal{L}^* \neq \mathcal{L}$; that is \(L_h^* \leq L_h \, , \;\;
h = 1, \ldots, m \, ,\) with $L_h^* < L_h$ in at least one case. }
\end{definition}
We can develop a geometrical approach to coherence by associating,
with each constituent $C_h$ contained in  $\mathcal{H}_n$, a point
$Q_h = (q_{h1}, \ldots, q_{hn})$, where $q_{hj} = 1$, or 0, or $p_j$, according to whether $C_h \subseteq E_jH_j$, or $C_h \subseteq E_j^cH_j$, or $C_h \subseteq H_j^c$.
Then, denoting by $\mathcal{I}$ the convex hull of $Q_1, \ldots, Q_m$,  the following characterization of coherence w.r.t. penalty criterion can be given (\cite[Theorem 4.4]{Gili90}, see also \cite{BiGS12,Gili92})
\begin{theorem}\label{CNES}{\rm
The function $P$ defined on $\F$ is coherent if and only if, for every finite sub-family $ \F_n
   \subseteq \mathcal{F}$, one has  $\P_n \in \I$. } \end{theorem}
\noindent{\em Equivalence between betting scheme and penalty criterion.}
The betting scheme and the penalty criterion are {\em
equivalent}, as can be proved by the following steps (\cite{Gili96}): \\
{\bf 1.} The condition $\P_n \in \I$ amounts to solvability of the
following system $(\Sigma)$ in the unknowns $\lambda_1, \ldots,
\lambda_m$
\[
(\Sigma) \hspace{1 cm}
\left\{
\begin{array}{l}
\sum_{h=1}^m q_{hj} \lambda_h = p_j \; , \; \; j = 1, \ldots, n \,
; \\[0.5ex]
\sum_{h=1}^m \lambda_h = 1 \; , \; \; \lambda_h \geq 0 \, ,
\; h = 1, \ldots, m.
\end{array}
\right.
\]
We say that system $(\Sigma)$ is  associated with the pair $(\mathcal{F}_n,\mathcal{P}_n)$.\\
{\bf 2.} Let ${\bf x} = (x_1, \ldots, x_m)$, ${\bf y} = (y_1, \ldots, y_n)^t$
and $A=(a_{ij})$ be, respectively, a row $m-$vector, a column
$n-$vector and a $m \times n-$matrix. The vector ${\bf x}$ is said
{\em semipositive} if $x_i \geq 0, \, \forall \, i, \;$ and $\;
x_1 + \cdots + x_m  >  0$. Then, we have (cf. \cite[Theorem 2.9]{Gale60})
\begin{theorem}\label{ALT1}{\rm
Exactly one of the following alternatives holds.
\\
(i) the equation ${\bf x} A = 0$ has a {\em semipositive} solution;
\\ (ii) the inequality $A {\bf y} > 0$ has a solution. }\end{theorem}
\noindent {\bf 3.} By choosing $a_{ij}=q_{ij}-p_j, \, \forall \, i,j$,
the solvability of ${\bf x} A = 0$ means that ${\P_n} \in \I$, while the
solvability of $A {\bf y} > 0$ means that, choosing $s_i = y_i, \,
\forall \, i$, one has $\min G_{\mathcal{H}_n} > 0$. Hence, by applying Theorem \ref{ALT1} with
$A=(q_{ij}-p_j)$, we obtain
$\max G_{\mathcal{H}_n}  \geq 0$ if and only if $(\Sigma)$ is solvable. In other words, $\max
G_{\mathcal{H}_n} \geq 0$ if and only if $\P_n \in \I$. Therefore, Definition \ref{COER-BET}  and Definition \ref{COER-PENALTY} are equivalent.
\subsection{Coherence Checking}
Given the assessment  $\mathcal{P}_n$ on $\mathcal{F}_n$, let $S$ be the set of solutions $\Lambda = (\lambda_1, \ldots,
\lambda_m)$ of the system $(\Sigma)$. Then, assuming $S \neq \emptyset$, define
\[\begin{array}{l}
\Phi_j(\Lambda) = \Phi_j(\lambda_1, \ldots, \lambda_m) = \sum_{r :
C_r \subseteq H_j} \lambda_r \; , \; \; \; j = 1, \ldots, n \,;\; \Lambda \in S \,;
\\
M_j  =  \max_{\Lambda \in S } \; \Phi_j(\Lambda) \; , \; \; \; j = 1, \ldots, n \,;
\;\;\; I_0  =  \{ j \, : \, M_j=0 \} \,.
\end{array}\]
We observe that, assuming $\P_n$ coherent, each solution $\Lambda=(\lambda_1, \ldots,
\lambda_m)$ of system $(\Sigma)$ is a coherent extension of the assessment $\mathcal{P}_n$ on $\mathcal{F}_n$ to the family $\{C_1|\H_n,\, \ldots,\,
C_m|\H_n\}$. Then, by the additive property, the quantity $\Phi_j(\Lambda)$ is the conditional probability $P(H_j|\H_n)$ and the quantity $M_j$ is the upper probability $P^*(H_j|\H_n)$ over all the solutions $\Lambda$ of system $(\Sigma)$. Of course, $j \in I_0$ if and only if $P^*(H_j|\H_n)=0$. Notice that $I_0 \subset \{1, \ldots, n\}$. We denote by $(\mathcal{F}_0, \mathcal{P}_0)$ the pair associated with $I_0$.
Given the pair $(\mathcal{F}_n,\mathcal{P}_n)$ and a subset $J \subset J_n=\{1,
\ldots, n\}$, we denote by $(\mathcal{F}_J, \mathcal{P}_J)$ the pair associated with
$J$ and by $\Sigma_J$ the corresponding system.
We observe that $(\Sigma_J)$ is solvable if and only if $\mathcal{P}_J  \in \mathcal{I}_J$,
where $\mathcal{I}_J$ is the convex hull associated with the pair $( \mathcal{F}_J,
\mathcal{P}_J)$. Then, we have  (\cite[Theorem 3.2]{Gili93}; see also \cite{BiGS03,Gili95})
\begin{theorem}\label{GILIO-93}{\rm
Given a probability assessment $\mathcal{P}_n$ on the family $\mathcal{F}_n$, if
the system $(\Sigma)$ associated with $(\mathcal{F}_n,\mathcal{P}_n)$ is solvable, then for every $J\subset \{1,\ldots,n\}$, such that $J\setminus I_0\neq \emptyset$, the system $(\Sigma_J)$ associated with $(\mathcal{F}_J,\mathcal{P}_J)$ is solvable too.}
\end{theorem}
The previous result says that the condition $\P_n \in \I$ implies $\P_J \in \I_J$ when $J\setminus I_0\neq \emptyset$. We observe that, if $\P_n \in \I$, then for every nonempty subset $J$ of $J_n\setminus I_0$ it holds that $J\setminus I_0=J \neq \emptyset$; hence, by Theorem \ref{CNES}, the subassessment $\P_{J_n\setminus I_0}$
on the subfamily $\F_{J_n\setminus I_0}$ is coherent. In particular, when $I_0$ is empty, coherence of $\P_n$ amounts to solvability of system $(\Sigma)$, that is to condition $\P_n \in \I$. When $I_0$ is not empty, coherence of $\P_n$ amounts to the validity of both conditions $\P_n \in \I$ and $\P_0$ coherent, as shown by the result below (\cite[Theorem 3.3]{Gili93})
\begin{theorem}\label{COER-P0}{\rm The assessment $\mathcal{P}_n$ on $\mathcal{F}_n$ is coherent if and only if the following conditions are satisfied: (i)
$\mathcal{P}_n \in \mathcal{I}$; (ii) if $I_0 \neq \emptyset$, then $\mathcal{P}_0$ is coherent.
}\end{theorem}
\subsection{Basic notions on probabilistic default reasoning}
Given a conditional knowledge base $\mathcal{K}_n = \{H_i
 \normally E_i \, , \; i=1,2,\ldots,n \}$, we denote by $\mathcal{F}_n = \{E_i|H_i \, , \; i=1,2,\ldots,n\}$ the associated
family of conditional events. We give below, in the setting of coherence, synthetic definitions of the notions of p-consistency and p-entailment of Adams, which are related with \cite[Theorem 4.5, Theorem 4.9]{BGLS02}, \cite[Theorem 5]{GiSa11b}, \cite[Theorem 6]{GiSa12}.
\begin{definition}\label{PC}
The knowledge base $\mathcal{K}_n = \{H_i \normally E_i \, , \; i=1,2,\ldots,n \}$ is {\em
\linebreak p-consistent} if and only if the assessment $(p_1,p_2,\ldots,p_n)=(1,1,\ldots,1)$ on $\mathcal{F}_n$ is coherent.
\end{definition}
\begin{definition}\label{PE}
A p-consistent knowledge  base  $\mathcal{K}_n = \{H_i \normally E_i
\, , \; i=1,\ldots,n \}$ {\em p-entails} the conditional $A
\normally B$, denoted $\mathcal{K}_n \; \Rightarrow_p \;A \normally
B$, if and only if, for every coherent assessment
$(p_1,p_2,\ldots,p_n,z)$ on $\mathcal{F}_n \cup \{B|A\}$ such that
$(p_1,p_2,\ldots,p_n)=(1,1,\ldots,1)$, it holds that $z=1$.
\end{definition}
The previous definitions of p-consistency and p-entailment are  equivalent (see \cite[Theorem 8]{Gili02}, \cite[Theorem 5]{GiSa11b}, \cite[Theorem 6]{GiSa12}) to that ones given in \cite{Gili02}.
\begin{remark}
\label{ENT-F}
We say that a family of conditional events $\mathcal{F}_n$ p-entails a conditional event $B|A$ when the associated knowledge  base  $\mathcal{K}_n$ p-entails the conditional $A |\hspace{-0.18cm}\sim B$.
\end{remark}
Definition \ref{PE} can be generalized to p-entailment of a family (of conditional events) $\Gamma$ from another family $\mathcal{F}$ in the following way.
\begin{definition}\label{ENTAIL-FAM}{\rm
Given two p-consistent finite families of conditional events
$\mathcal{F}$ and $\mathcal{S}$, we say that $\mathcal{F}$ p-entails
$\mathcal{S}$ if $\mathcal{F}$ p-entails $E|H$ for every $E|H \in
\mathcal{S}$. }\end{definition} We remark that, from Definition \ref{PE},
we trivially have that $\mathcal{F}$ p-entails $E|H$, for every $E|H
\in \mathcal{F}$; then, by Definition \ref{ENTAIL-FAM}, it
immediately follows
\begin{equation}\label{ENTAIL-SUB}
\mathcal{F} \;\Rightarrow_p \; \mathcal{S} \;,\;\; \forall \, \mathcal{S} \subseteq \mathcal{F} \,.
\end{equation}
Probabilistic default reasoning has been studied by many authors (see, e.g.,\cite{BeDP97,BGLS02,BGLS05,DuPr94,KrLM90}); methods and results based on the maximum entropy principle have been given in \cite{Kern01,TKIF11}.
\subsection{Hamacher t-norm and t-conorm}
We recall that the Hamacher t-norm, with parameter $\lambda=0$, or Hamacher product, $T_0^H$ is defined as (\cite{Hama78})
\begin{equation}\label{HAM-T0} T_0^H(x,y) =
\left\{\begin{array}{ll} 0, & (x,y)=(0,0),
\\ \frac{xy}{x+y-xy}, & (x,y)\neq (0,0).
\end{array}\right.
\end{equation}
We also recall that the Hamacher t-conorm, with parameter $\lambda=0$, $S_0^H$ is
\begin{equation}\label{HAM-S0} S_0^H(x,y) = \left\{\begin{array}{ll} 1, & (x,y)=(1,1),
\\ \frac{x+y-2xy}{1-xy}, & (x,y)\neq(1,1).
\end{array}\right.
\end{equation}
As is well known, t-norms overlap with copulas (\cite{AlFS06,Nels99}); indeed, commutative associative copulas are t-norms and t-norms which satisfy the 1-Lipschitz condition are copulas. We also recall that some well-known families of t-norms receive a different name in the literature when considered as families of copulas. In particular,  the Hamacher product is a copula because it satisfies the following necessary and sufficient condition (\cite[Theorem 1.4.5]{AlFS06}):
\begin{theorem}
A t-norm T is a copula if and only if it satisfies the Lipschitz condition:
$ T(x_2,y)-T(x_1,y)\leq x_2-x_1$, whenever  $x_2\leq x_1$.
\end{theorem}
Hamacher product is called Ali-Mikhail-Haq copula with parameter 0 (\cite{AlMH78,GMMP09,KlMP00,Nels99}).
Further details on  t-norms and t-conorms are given in the  Appendices.
\section{Lower and Upper Bounds for Quasi Conjunction}
We recall below the notion of quasi conjunction of conditional events  as defined in \cite{Adam75}.
\begin{definition}\label{QC}{\rm
Given any events $A, H$, $B, K$, with $H \neq
\emptyset, K \neq \emptyset$, the quasi conjunction of the
conditional events $A|H$ and $B|K$ is
the conditional event $\mathcal{C}(A|H,B|K) = (AH \vee H^c) \wedge (BK \vee K^c)|(H \vee K)$, or equivalently $\mathcal{C}(A|H,B|K) = (AHBK \vee AHK^c \vee H^cBK)|(H \vee K)$.
}\end{definition}
Table \ref{TABLE-QC} shows the truth-table of the quasi conjunction $\C(A|H,B|K)$ and of the
quasi disjunction $\D(A|H,B|K)$ (see Section \ref{SEC-QD}).
In general, given a family of $n$ conditional events $\mathcal{F}_n=\{E_i|H_i,\, i=1,\ldots,n\}$, we have
\[
\mathcal{C}(\mathcal{F}_n)=\mathcal{C}(E_1|H_1,\ldots,E_n|H_n) = \bigwedge_{i=1}^n (E_iH_i\vee H_i^c)\big |(\bigvee_{i=1}^n H_i) \,.
\]
Quasi conjunction is associative; that is, for every subset $J \subset \{1,\ldots,n\}$, it holds that
$\mathcal{C}(\mathcal{F}_n) = \mathcal{C}(\mathcal{F}_J \cup \mathcal{F}_{\Gamma}) = \mathcal{C}[\mathcal{C}(\mathcal{F}_J), \mathcal{C}(\mathcal{F}_{\Gamma})]$, where $\Gamma = \{1,\ldots,n\} \setminus J$.
An interesting analysis of many three-valued logics studied in the literature has been given by Ciucci and Dubois in \cite{CiDu12}. In such a paper the definition of conjunction satisfies left monotonicity, right monotonicity and conformity with Boolean logic; then the authors show that there are 14 different ways of defining conjunction and only 5 of them  (one of which defines quasi conjunction) satisfy commutativity and associativity.\\
 Assuming $A, H, B, K$ logically independent, we have
(\cite{Gili04}, see also \cite{Gili12}): \\ (i) the probability assessment $(x,y)$ on $\{A|H, B|K\}$ is coherent for every $(x,y) \in [0,1]^2$; \\ (ii) given a coherent
assessment $(x,y)$ on $\{A|H, B|K\}$, the extension $\mathcal{P} = (x,y,z)$ on
$\mathcal{F} = \{A|H, B|K, \mathcal{C}(A|H,B|K)\}$, with $z = P[\mathcal{C}(A|H,B|K)]$, is
coherent if and only if $z\in[l,u]$, with
\begin{equation} \label{EQ-LOW-UP-QC}
l = T_L(x,y) = \mbox{max}(x+y-1,0),\;\;  u = S_0^H(x,y) = \left\{\begin{array}{ll} \frac{x+y-2xy}{1-xy}, & (x,y) \neq
(1,1), \\ 1, & (x,y)=(1,1),
\end{array} \right.
\end{equation}
where $T_L$ is the Lukasiewicz t-norm (see \ref{T-NORM}) and $S_0^H$ is the Hamacher t-conorm\footnote{\noindent The coincidence between the upper bound $u$ and Hamacher t-conorm $S_0^H(x,y)$ was noticed by Didier Dubois.}, with parameter $\lambda=0$. The lower bound $T_L$ for the quasi conjunction is the Fr\'echet-Hoeffding lower bound; both $l$ and $u$ coincide with the Fr\'echet-Hoeffding bounds if we consider the {\em conjunction} of conditional events in the setting of conditional random quantities, as made in \cite{GiSa13a}.
The lower and upper bounds, $l,u$, of $z=P[\mathcal{C}(A|H,B|K)]$ can be obtained by studying the coherence of the assessment $\P=(x,y,z)$, based on the geometrical approach described in Section \ref{PRELIM}.
The constituents generated  by the family
$\{A|H, B|K, \C(A|H,B|K)\}$ and the corresponding points $Q_h$'s
are given in columns 2 and 6  of Table \ref{TABLE-QC}.
In \cite{Gili04} (see also \cite{Gili12})  the values $l,u$ are computed by observing that the coherence of $\P=(x,y,z)$ simply amounts to the geometrical condition $\P\in \I$, where $\I$ is the convex hull of the points $Q_1,Q_2,\ldots,Q_8$ (associated with the constituents $C_1,C_2,\ldots,C_8$ contained in $H\vee K$). We observe that in this case the convex hull $\I$ does not depend on $z$.
Figure~\ref{fig1} shows, for  given values $x,y$, the convex hull $\I$  and the associated interval $[l,u]$ for $z=P[\C(A|H,B|K)]$.
\begin{table}[!ht]
\center
\[
\begin{array}{cc}
\begin{small}
\begin{array}{l|llccccc}
h & C_h       & A|H       &  B|K   & \mathcal{C}(A|H,B|K)&Q_h   & \mathcal{D}(A|H,B|K)    &Q_h   \\ \hline
0&H^cK^c      &    \Void  &  \Void & \Void  & (x,y,z)  &\Void    &   (x,y,z)\\
1&AHBK        &    \True  &  \True & \True  & (1,1,1)  & \True   &   (1,1,1)     \\
2&AHK^c       &    \True  &  \Void & \True  &  (1,y,1) & \True   &   (1,y,1)      \\
3&AHB^cK      &    \True  &  \False& \False & (1,0,0)  & \True   &   (1,0,1)     \\
4&H^cBK       &    \Void  &  \True & \True  &  (x,1,1) & \True   &   (x,1,1)      \\
5&H^cB^cK     &    \Void  &  \False& \False & (x,0,0)  & \False  &   (x,0,0)    \\
6&A^cHBK      &    \False &  \True & \False & (0,1,0)  & \True   &   (0,1,1)     \\
7&A^cHK^c     &    \False &  \Void & \False & (0,y,0)  & \False  &   (0,y,0)     \\
8& A^cHB^cK   &    \False &  \False& \False & (0,0,0)  & \False  &   (0,0,0)
  \end{array} 
\end{small}
\end{array}
\]
\caption{Truth-Table of the quasi conjunction and of the quasi disjunction with the associated points $Q_h$'s. }
\label{TABLE-QC}
\end{table}
\begin{figure}[!ht]
\begin{center}
\includegraphics[ scale=\myscale]{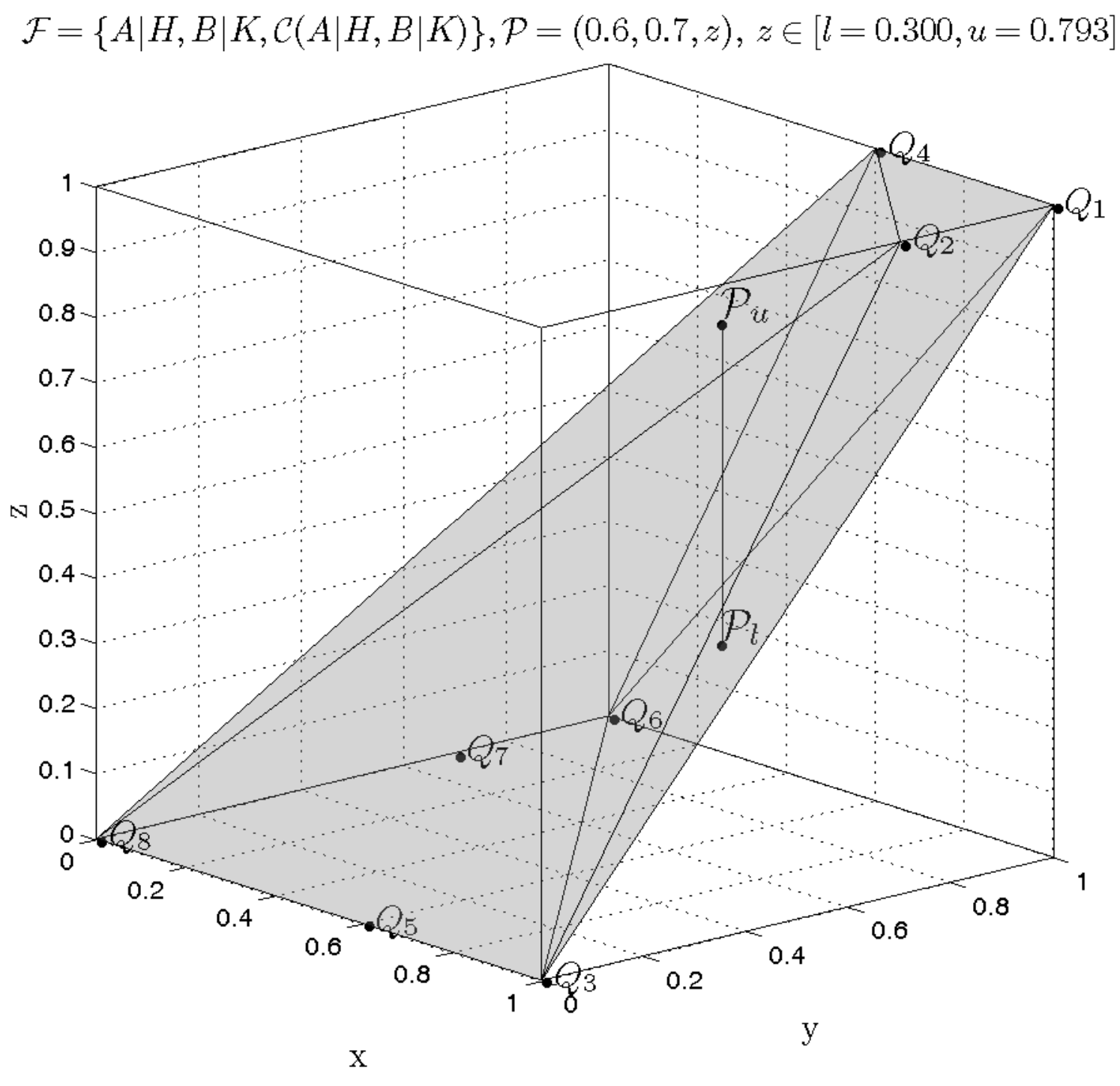}
\caption{The convex hull $\I$  associated with the pair $(\F,\P)$ in case of quasi conjunction without logical dependencies.
The  interval $[l,u]$ for $z=P[\C(A|H,B|K)]$ is the range of the third coordinate $z$  of each $\P \in
     \overline{\P_l\P_u}=\{(x,y,z):z\in[T_L(x,y),S_{0}^{H}(x,y)]\}$. The segment $\overline{\P_l\P_u}$  is the intersection  between the segment $\{(x,y,z):z\in [0,1]\}$  and the convex hull $\I$.}
\label{fig1}
\end{center}
\end{figure}
\begin{remark}\label{REM-SC}
Notice that, if the events $A, B, H, K$ were not logically independent, then some constituents $C_h$'s (at least one) would become impossible and the lower bound $l$ could increase, while the upper bound $u$ could decrease. To examine this aspect we will consider some special cases of logical dependencies.
\end{remark}
\subsection{The Case $A|H\subseteq B|K$}\label{INCLU}
The notion of logical inclusion among events has been generalized to conditional events by Goodman and Nguyen in \cite{GoNg88}. We recall below this generalized notion.
\begin{definition}{\rm Given two conditional events $A|H$ and $B|K$, we say that $A|H$ implies $B|K$, denoted by $A|H \subseteq B|K$, iff $AH$ {\em true} implies $BK$ {\em true} and $B^cK$ {\em true} implies $A^cH$ {\em true}; i.e., iff $AH \subseteq BK \; \mbox{and} \; B^cK \subseteq A^cH$.
}\end{definition}
\begin{remark}\label{REM-GN}
{\rm
Denoting by $t(\cdot)$ the truth value function and assuming the order $False<Void <True$, then  it can be easily verified that
\[
\begin{array}{l}
A|H \subseteq B|K \Leftrightarrow AHB^cK=H^cB^cK=AHK^c=\emptyset\,, \\
A|H \subseteq B|K\Leftrightarrow t(A|H)\leq t(B|K) \Leftrightarrow  t(B^c|K)\leq t(A^c|H) \Leftrightarrow B^c|K \subseteq A^c|H\,.
\end{array}
\]
}\end{remark}
Given any conditional events $A|H,B|K$, we denote by  $\Pi_x$  the set of coherent probability assessments $x$ on $A|H$,  by  $\Pi_y$  the set of coherent probability assessments $y$ on $B|K$ and by
$\Pi$  the set of coherent probability assessments $(x,y)$ on $\{A|H, B|K \}$; moreover we indicate by $T_{x\leq y}$  the triangle $\{(x,y) \in [0,1]^2 : x \leq y\}$. We have
\begin{theorem}\label{COER-GN}
{\rm
Given two conditional events $A|H,B|K$, we have
\begin{equation}\label{EQ-COER-GN}
\Pi \subseteq T_{x\leq y} \; \Longleftrightarrow \; A|H \subseteq B|K, \; \mbox{or} \; AH=\emptyset, \; \mbox{or} \; B^cK=\emptyset \,.
\end{equation}
}\end{theorem}
\begin{proof}
($\Rightarrow$) We will prove that
\begin{equation}\label{NEG-EQ-COER-GN}
A|H \nsubseteq B|K \,,\; AH \neq \emptyset \,,\; B^cK \neq \emptyset \; \Longrightarrow \; \Pi \nsubseteq T_{x\leq y} \,.
\end{equation}
We observe that $AH = \emptyset$ if and only if $\Pi_x = \{0\}$ and that $B^cK = \emptyset$ if and only if $\Pi_y = \{1\}$. Moreover, by Remark \ref{REM-GN} it holds
\[
A|H \nsubseteq B|K \; \Longleftrightarrow \; AHB^cK \vee H^cB^cK \vee AHK^c \neq \emptyset \,.
\]
Then, in order to prove formula (\ref{NEG-EQ-COER-GN}), we distinguish three cases: \\
(i) $AHB^cK \neq \emptyset$;  (ii) $H^cB^cK \neq \emptyset \,,\; AH \neq \emptyset$;  (iii) $AHK^c \neq \emptyset \,,\; B^cK \neq \emptyset$. \\
In case (i), the assessment $(1,0)$ on $\{A|H, B|K\}$ is coherent. In case (ii), as $AH \neq \emptyset$ we have $\{1\} \subseteq \Pi_x$; then, the assessment $(1,0)$ on $\{A|H, B|K\}$ is coherent. In case (iii), as $B^cK \neq \emptyset$ we have $\{0\} \subseteq \Pi_y$; then, the assessment $(1,0)$ on $\{A|H, B|K\}$ is coherent. Then, in each of the three cases the assessment $(1,0)$ is coherent and hence $\Pi \nsubseteq T_{x\leq y}$. \\
($\Leftarrow$)
We distinguish three cases: \\ $(a)$ $A|H \subseteq B|K$;  $(b)$ $AH=\emptyset$; $(c)$ $B^cK=\emptyset$.\\
$(a)$ The constituents generated by $\{A|H, B|K\}$ and contained in $H\vee K$ belong to the family:
\[\{AHBK, H^cBK, A^cHBK, A^cHK^c, A^cHB^cK\}\,.\]
The corresponding points $Q_h$'s belong to the set
$\{(1,1), (x,1), (0,1), (0,y), (0,0)\}$, which has the triangle $T_{x\leq y}$ as convex hull; hence the convex hull $\Pi$ of the points $Q_h$'s is a subset of $T_{x\leq y}$.
 \\
$(b)$ Since $\Pi_x=\{0\}$ it follows that   $\Pi\subseteq \{(0,y): y\in [0,1]\}\subseteq T_{x\leq y}$.
\\
$(c)$ Since $\Pi_y=\{1\}$ it follows that   $\Pi\subseteq \{(x,1): x\in [0,1]\}\subseteq T_{x\leq y}$.
\end{proof}
The next result, related with Theorem \ref{COER-GN} and with the inclusion relation,  characterizes the notion of p-entailment between two conditional events.
\begin{theorem}\label{THM-ENTAIL-COER-GN}
Given two conditional events $A|H, B|K$, with $AH\neq \emptyset$, the following assertions are equivalent: \\
(a) $(A|H \;\Rightarrow_p \; B|K )\;$;
(b) $A|H \;\subseteq \, B|K$, or $K\subseteq B \;$;
(c) $\Pi\subseteq T_{x\leq y}$.
\end{theorem}
\begin{proof}
As $AH \neq \emptyset$, from Theorem \ref{COER-GN} the assertions (b) and (c) are equivalent; hence, we only need to show the equivalence between (a) and (b). \\
((a) $\Rightarrow$ (b)).
We will prove that
\[
A|H \nsubseteq B|K \,,\; AH \neq \emptyset \,,\; B^cK \neq \emptyset \; \; \Longrightarrow \; A|H \;\nRightarrow_p \; B|K \,.
\]
Assume that $A|H \nsubseteq B|K, B^cK \neq \emptyset$. Then, as in the proof of Theorem \ref{COER-GN}, we distinguish three cases: \\
(i) $AHB^cK \neq \emptyset$;  (ii) $H^cB^cK \neq \emptyset \,,\; AH \neq \emptyset$;  (iii) $AHK^c \neq \emptyset \,,\; B^cK \neq \emptyset$. \\
In all three cases the assessment $(1,0)$ is coherent; thus $A|H \nRightarrow_p B|K$. \\
((b) $\Rightarrow$ (a)). We preliminarily observe that $\{A|H\}$ is p-consistent. Now, if $A|H \;\subseteq \, B|K$, then p-entailment of $B|K$ from $A|H$ follows from monotonicity of conditional probability w.r.t. inclusion relation. If $K\subseteq B$, then trivially $A|H$ p-entails $B|K$.
\end{proof}
\begin{example} Given any events $A, B$, for the conditional events $A \vee B, B|A^c$ it holds that $B|A^c = (A \vee B)|A^c \subseteq A \vee B$. Then, for the assessment $P(A \vee B)=x, P(B|A^c)=y$, the necessary condition of coherence $0 \leq y \leq x \leq 1$ must be satisfied.
 Of course, $P(A \vee B)$ 'high' does not imply $P(B|A^c)$ 'high'; for instance, it is coherent to assign $P(B|A^c)=0.2$ and $P(A \vee B)=0.8$. Then, the inference of the conditional event $B|A^c$ from the disjunction $A \vee B$ may be 'weak'. A probabilistic analysis characterizing the cases in which such an inference is 'strong' has been made in \cite{GiOv12}.
\end{example}
\begin{remark}\label{GN-REM}
We observe that, under the hypothesis $A|H\subseteq B|K$, the constituents generated by  $\{A|H,B|K\}$ belong to the family
\[
\mathcal{H}=\{AHBK, H^cBK, A^cHBK, A^cHK^c, A^cHB^cK, H^cK^c\} \,.
\]
The quasi-conjunction is $\mathcal{C}(A|H,B|K)=(AH \vee H^cBK) \,|\, (H \vee K)$ and, as shown by  Table \ref{TABLE-QC}, for any constituent in $\mathcal{H}$ it holds that
\[
t(A|H)\leq t(\mathcal{C}(A|H,B|K) )\leq  t(B|K) \,.
\]
Then, we have (see Remark \ref{REM-GN})
\begin{equation}\label{QC-GN}
A|H \subseteq B|K \; \Longrightarrow \;  A|H \subseteq \mathcal{C}(A|H,B|K) \subseteq B|K \,.
\end{equation}
\end{remark}
As conditional probability is monotonic w.r.t. inclusion relation among conditional events (\cite{GoNg88}), it holds that $P(A|H) \leq P[\mathcal{C}(A|H,B|K)] \leq P(B|K)$. As shown by Theorem \ref{COER-GN}, in our coherence-based approach the monotonic property is obtained without assuming that $P(H)$ and $P(K)$ are positive. The next result establishes that  $P[\mathcal{C}(A|H,B|K)]$ can coherently assume all the values in the interval $[P(A|H),P(B|K)]$. We have
\begin{proposition}\label{PROP-QC-GN}{\rm Let be given any coherent assessment $(x,y)$ on $\{A|H, B|K\}$, with $A|H \subseteq B|K$ and with no further logical relations. Then, the extension $z = P[\mathcal{C}(A|H,B|K)]$ is coherent if and only if $l \leq z \leq u$, where
\[
 l = x = \min(x,y)=T_M(x,y) \,,\; u = y = \max(x,y)=S_M(x,y) \,.
\]
}\end{proposition}
\begin{proof}  We recall that, apart from $A|H \subseteq B|K$, there are no further logical relations; thus it holds that $\Pi = T_{x \leq y}$ (i.e. $0\leq x\leq y\leq 1$). Denoting by $[l,u]$ the interval of coherent extensions of the assessment $(x,y)$ to $\mathcal{C}(A|H,B|K)$, by (\ref{QC-GN}) it holds that $[l,u] \subseteq [x,y]$. In order to prove that $[l,u] = [x,y]$ it is enough to verify that both the assessments $\P_l=(x,y,x)$ and $\P_u=(x,y,y)$ are coherent. Given any assessment $\P=(x,y,z)$, with $x\leq y$, we study the coherence of $\P$ by the geometrical approach described in Section \ref{PRELIM}. The constituents generated by the family and contained in $H\vee K$ are:
\[
\begin{array}{l}
C_1 = AHBK \,,   C_2 = H^cBK \,,
  C_3 = A^cHBK \,,  C_4 = A^cHK^c \,,  C_5
= A^cHB^cK  \,.
\end{array}
\]
The corresponding points $Q_h$'s are
\[
\begin{array}{l}
Q_1 = (1,1,1) \,,  Q_2 =
(x,1,1)\,,    Q_3 = (0,1,0)\,, Q_4 =
(0,y,0) \,,  Q_5 = (0,0,0)  \,,
\end{array}
\]
and, in our case, the coherence of $\P$ simply amounts to the
geometrical condition $\P \in \I$, where $\I$ is the convex hull
of the points $Q_1,Q_2,\ldots, Q_5$. \\ It can be verified that $\P_l = x Q_1 + (y-x)Q_3 + (1-y)Q_5$, so that $\P_l \in \I$; hence $l=x$. Concerning $\P_u$, we first observe that when $(x,y)=(1,1)$ we have $\P_u=(1,1,1)=Q_1$, so that $\P_u \in \I$; hence $u=y=1$. Assuming $(x,y) \neq (1,1)$, it can be verified that $\P_u = \frac{x-xy}{1-x} Q_1 + \frac{y-x}{1-x}Q_2 + (1-y)Q_5$, so that $\P_u \in \I$; hence $u=y$.
Therefore, $[l,u]=[x,y]$.
\end{proof}
We remark that the lower/upper bound above, $l,u$, may change if we add further logical relations; for instance, if $H=K$, it is $\mathcal{C}(A|H,B|H) = A|H$, in which case $l=u=x$.
Finally, in agreement with Remark \ref{REM-SC}, we observe that
$T_L(x,y) \leq  \min(x,y) \leq \max(x,y) \leq S_0^H(x,y)$.
We also recall that  $T_M(x,y)=\min(x,y)$ is the largest t-norm and $S_M(x,y)=\max(x,y)$ is the smallest t-conorm (\cite{KlMP05}).
Figure \ref{fig4}  shows the convex hull $\I$ for given values  $x,y$, with the associated interval $[l,u]$ for $z=P[\C(A|H,B|K)]$, when $A|H\subseteq B|K$.
\begin{figure}[!ht]
\begin{center}
\includegraphics[ scale=\myscale]{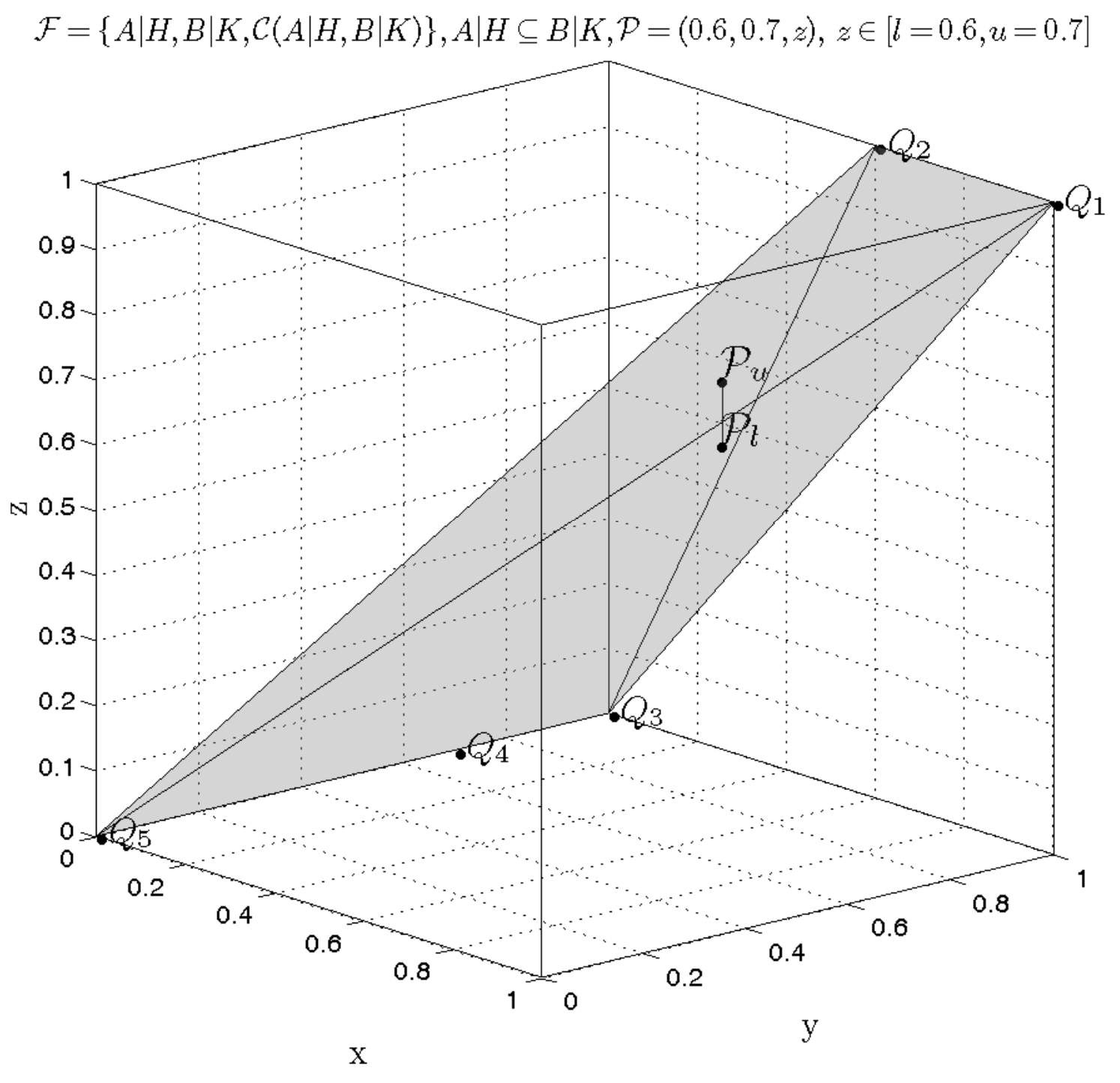}
\caption{The convex hull $\I$  associated with the pair $(\F,\P)$ when $A|H\subseteq B|K$.
The  interval $[l,u]$ for $z=P[C(A|H,B|K)]$ is the range of the third coordinate $z$  of each $\P \in \overline{\P_l\P_u}=\{(x,y,z):z\in[T_M(x,y),S_{M}(x,y)]\}$. The segment $\overline{\P_l\P_u}$ is the intersection  between the segment $\{(x,y,z):z\in [0,1]\}$  and the convex hull $\I$. This  intersection  is empty for $x>y$ because of $\Pi\subseteq T_{x\leq y}$, .}
\label{fig4}
\end{center}
\end{figure}
\subsection{Compound Probability Theorem}\label{SUB-COMPOUND}
We now examine the quasi conjunction of $A|H$ and $B|AH$, with  $A,\,B,\, H$ logically independent events.
As it can be easily verified, we have $\mathcal{C}(A|H,B|AH) = AB|H$; moreover, the probability assessment $(x,y)$ on $\{A|H, B|AH \}$ is coherent, for every $(x,y)\in[0,1]^2$. Hence, by the compound probability theorem, if the  assessment $\mathcal{P} = (x,y,z)$ on
$\mathcal{F} =$\linebreak $\{A|H, B|AH, AB|H\}$ is coherent, then  $z=xy$; i.e., $l=u=x\cdot y=T_P(x,y)$. \\
In agreement with Remark~\ref{REM-SC}, we observe that $T_L(x,y) \leq  xy \leq S_0^H(x,y)$. \\
We observe that $A|H = AH|H, B|AH = ABH|AH, AB|H = ABH|H$; as $z=xy$, $\{AH|H, ABH|AH\}$ p-entails $ABH|H$ (\textit{transitive property}). Moreover $AB|H \subseteq B|H$; hence $\{A|H, B|AH\}$ p-entails $B|H$ ({\em cut rule}). \\

\subsection{Lower and Upper Bounds for the Quasi Conjunction of $n$ Conditional Events}
\label{SEC-QCN}
In this subsection we generalize formula (\ref{EQ-LOW-UP-QC}). Let be given $n$ conditional events $E_1|H_1, \ldots, E_n|H_n$. By the associative property of quasi conjunction, defining $\mathcal{F}_k = \{E_1|H_1, \ldots, E_k|H_k\}$, for each $k=2,
\ldots,n$ it holds that $\mathcal{C}(\mathcal{F}_k) = \mathcal{C}(\mathcal{C}(\mathcal{F}_{k-1}), E_k|H_k)$. Then, we have
\begin{theorem}\label{LOW-UP-QC} {\rm Given a probability assessment
$\mathcal{P}_n = (p_1, p_2, \ldots, p_n)$ on $\mathcal{F}_n = \{E_1|H_1, \ldots,
E_n|H_n\}$, let $[l_k,u_k]$ be
the interval of coherent extensions of the assessment $\mathcal{P}_k = (p_1, p_2, \ldots, p_k)$ on
the quasi conjunction $\mathcal{C}(\mathcal{F}_k)$, where $\mathcal{F}_k = \{E_1|H_1, \ldots,
E_k|H_k\}$. Then, assuming $E_1, H_1, \ldots, E_n, H_n$ logically independent,
for each $k=2,\ldots,n$, we have
\begin{equation}\label{LOWER-Lk}
l_k = T_L(p_1,p_2, \ldots, p_k)= \max(p_1+p_2+\ldots+p_k-(k-1), 0)  \,,
\end{equation}
\begin{equation}\label{UPPER-Uk}
u_k = S_0^H(p_1,p_2, \ldots,p_k) = \left\{\begin{array}{ll} 1, &  p_i=1 \mbox{ for at least one } i,
\\
\frac{\sum_{i=1}^k\frac{p_i}{1-p_i}
}
{
\sum_{i=1}^k\frac{p_i}{1-p_i}+1
}
 \,, & p_i< 1 \mbox{ for } i=1,\ldots, k.
\end{array}\right.
\end{equation}
}
\end{theorem}
\begin{proof}
 Of course, from (\ref{EQ-LOW-UP-QC}) it is $l_2 =
T_L(p_1,p_2) \,,\;\;\; u_2 =
S_0^H(p_1,p_2)$.
We recall that both $T_L$ and $S_0^H$ are associative.
Moreover, as
\[
\C(\F_3) = \C(\C(\F_2), E_3|H_3) \,,\;\;\; l_2 \leq P[\C(\F_2)]
\leq u_2 \,, \] defining $P[\C(\F_2)] = x$ and observing that the
quantities
$ T_L(x,p_3) $ , $S_0^H(x,p_3)$
are non-decreasing functions of $x$, we have
\[\begin{array}{l}
l_3 =T_L(l_2,p_3)=T_L(T_L(p_1,p_2),p_3)= T_L(p_1,p_2,p_3) \,, \end{array}\]
\[\begin{array}{l}
u_3 = S_0^H(u_2,p_3)=S_0^H(S_0^H(p_1,p_2),p_3)= S_0^H(p_1,p_2,p_3)
 \,.
\end{array}\]
Considering any $k > 3$, we proceed by induction. Assuming
\[
\begin{array}{l}
l_{k-1} = T_L(p_1,p_2,\ldots,p_{k-1}) \,, \;\;\; u_{k-1}= S_0^H(p_1,p_2,\ldots,p_{k-1}) \,,
\end{array}\]
as $\C(\F_k) = \C(\C(\F_{k-1}), E_k|H_k)$ and  $l_{k-1} \leq P[\C(\F_{k-1})] \leq u_{k-1}$,
 defining $P[\C(\F_{k-1})] = x$ and observing that the quantities $ T_L(x,p_k)$ and $S_0^H(x,p_k)$
 are non-decreasing functions of $x$, we have
\[\begin{array}{l}
l_k =T_L(l_{k-1},p_k)=T_L(p_1,p_2,\ldots,p_{k}), \\ u_k = S_0^H(u_{k-1},p_k)=S_0^H(p_1,p_2,\ldots,p_{k}).
\end{array}\]
The explicit values of $l_k$ and $u_k$ in (\ref{LOWER-Lk}) and (\ref{UPPER-Uk}) follow by \ref{SEZ-TNORMN} and \ref{HAM-NORM-CON}.
\end{proof}
\noindent Notice that $(p_1,p_2,\ldots,p_n) = (1,1,\ldots,1)$ implies $T_L(p_1,p_2,\ldots,p_n) =1$. Then, from Theorem \ref{LOW-UP-QC}, we obtain the following {\em Quasi And} rule
 (see also \cite[Theorem 4]{GiSa11b}).
\begin{corollary}\label{ENT-QC}{\rm
Given a p-consistent family of conditional events $\mathcal{F}_n$, we have
\begin{equation}\label{ENTAIL-QCSUB}
\begin{array}{llll}
\mbox{(Quasi And)} &  \hspace{2.5cm} & \F_n \Rightarrow_p  \C(\F_n) \;. &\hspace{3.5cm}
\end{array}
\end{equation}
}\end{corollary}
 We observe that, from (\ref{ENTAIL-SUB}), we obtain (\cite[Theorem 5]{GiSa12})
\begin{equation}\label{IJAR-THM5}
\mathcal{F}_n \Rightarrow_p  \C(\mathcal{S})\;,\;\; \forall  \mathcal{S}\subseteq \mathcal{F}_n\,.
\end{equation}
Of course,   (\ref{IJAR-THM5}) still holds when there are  logical dependencies because in this case the lower bound for quasi conjunction does not decrease, as observed in Remark \ref{REM-SC}. In the next example we illustrate the key role of quasi conjunction  when we study
p-entailment. This example has been already examined in \cite{Gili02},  by using the inference rules of System P in the setting of coherence.
\begin{example}[{\em Linda's example}]
We start with a given p-consistent family of conditional
events $\mathcal{F}$; then, we use the quasi conjunction to check the
p-entailment of some further conditional events from $\mathcal{F}$. The
family  $\mathcal{F}$ concerns various attributes for a given party  (\emph{the party is great, the party is noisy, Linda and Steve are present}, and so on). We introduce the following events:
\[
\begin{array}{ll}
L= \mbox{``\emph{Linda goes to the party}''}; & S= \mbox{``\emph{Steve goes to the party}''};\\
G= \mbox{``\emph{the party is great}''}; & N= \mbox{``\emph{the party is noisy}''}\,,
\end{array}
\]
which are assumed to be  logically independent. Then, we consider the family $\mathcal{F}=\{G|L,S|L,N^c|LS,L|S,G^c|N^c\}$ and the family of further conditional events
\[
\mathcal{K}=\{N^c|L,\;  L^c|\Omega,\;  GN^c|LS,\;  N^c|S,\;  N^c|(L\vee S)\}\,.
\]
It can be verified that the assessment $(1,1,1,1,1)$ on $\F$ is coherent, i.e. the family $\F$ is p-consistent. By exploiting quasi conjunction, we can verify that $\F$ p-entails $\K$; that is $\mathcal{F}$ p-entails each conditional event in $\K$. Indeed:\\
(a) concerning $N^c|L$, for the subset $\S=\{S|L,N^c|LS\}$ we have $\C(\S)=N^cS|L \subseteq N^c|L$; thus: $\F \Rightarrow_p \C(\S)  \Rightarrow_p N^c|L;$\\
(b) concerning $L^c|\Omega$, for the subset $\S=\{G|L,S|L,N^c|LS,G^c|N^c\}$ we have $\C(\S)=G^cL^cN^c|(L\vee N^c) \subseteq L^c|\Omega $; thus: $\F \Rightarrow_p \C(\S)  \Rightarrow_p L^c|\Omega;$\\
(c) concerning $GN^c|LS$, for the subset $\S=\{G|L,S|L,N^c|LS\}$ we have $\C(\S)
=GN^cS|L \subseteq GN^c|LS $; thus: $\F \Rightarrow_p \C(\S)  \Rightarrow_p GN^c|LS;$\\
(d) concerning $N^c|S$, for the subset $\S=\{N^c|LS,L|S\}$ we have $\C(\S)
=LN^c|S \subseteq N^c|S $; thus: $\F \Rightarrow_p \C(\S)  \Rightarrow_p N^c|S;$\\
(e) concerning $N^c|(L\vee S)$, for the subset $\S=\{S|L,N^c|LS,L|S\}$ we have $\C(\S)
=LN^cS|(L\vee S) \subseteq N^c|(L\vee S)$; thus: $\F \Rightarrow_p \C(\S)  \Rightarrow_p N^c|(L\vee S)$.\\
We point out  that the p-entailment of  $\K$ from $\F$ can be also verified by applying  Algorithm 2 in \cite{GiSa12}. We also observe that,  using the  basic events $L,S,G,N$, we can define conditional events which are not p-entailed from $\F$. For instance, concerning   $G|N$, associated with the conditional   ``\emph{if the party is noisy, then the party is great}'', it can be proved that $\F$ does not p-entail $G|N$. Indeed, there is no  subset $\S\subseteq \F$, with $\S\neq \emptyset$,  such that $ \C(\S)  \Rightarrow_p G|N$  (see \cite[Theorem 6]{GiSa12}).
\end{example}
\subsection{The Case $E_1|H_1\subseteq E_2|H_2\subseteq \ldots \subseteq E_n|H_n$}
In this subsection we give a result on quasi conjunctions when  $E_i|H_i \subseteq E_{i+1}|H_{i+1}, i=1,\ldots,n-1$. We have
\begin{theorem}\label{THM-GN}
Given a family $\mathcal{F}_n = \{E_1|H_1, \ldots, E_n|H_n\}$ of conditional events such that
$E_1|H_1\subseteq E_2|H_2\subseteq \ldots \subseteq E_n|H_n$,
and a coherent probability assessment
$\mathcal{P}_n = (p_1, p_2, \ldots, p_n)$ on $\mathcal{F}_n$,
let $\mathcal{C}(\mathcal{F}_k)$ be the quasi conjunction of $\mathcal{F}_k=\{E_i|H_i, i=1,\ldots,k\}$, $k=2,\ldots,n$. Moreover, let $[l_k,u_k]$ be the interval of coherent extensions on $\mathcal{C}(\mathcal{F}_k)$ of the assessment $(p_1, p_2, \ldots, p_k)$ on $\mathcal{F}_k$. We have:
(i) $E_1|H_1\subseteq \mathcal{C}(\mathcal{F}_2) \subseteq \ldots \subseteq \mathcal{C}(\mathcal{F}_n) \subseteq E_n|H_n$;
(ii) by assuming no further logical relations, any probability assessment $(z_2,\ldots,z_k)$ on $\{\mathcal{C}(\mathcal{F}_2), \ldots, \mathcal{C}(\mathcal{F}_k)\}$ is a coherent extension of the assessment $(p_1, p_2, \ldots, p_k)$ on $\mathcal{F}_k$ if and only if
$p_1 \leq z_2 \leq \cdots \leq z_k \leq p_k \,,\; k=2,\ldots,n$;
moreover
\[l_k = \min (p_1,\ldots, p_k) = p_1 \,,\;  u_k = \max (p_1,\ldots, p_k) = p_k \,,\; k=2,\ldots,n \,. \]
\end{theorem}
\begin{proof} (i) By iteratively applying (\ref{QC-GN}) and by the associative property of quasi conjunction, we have
$\mathcal{C}(\mathcal{F}_{k-1}) \subseteq \mathcal{C}(\mathcal{F}_k) \subseteq E_k|H_k \,,\; k=2, \ldots, n$; \\
(ii) by exploiting the logical relations in point (i), the assertions immediately follow by applying a reasoning similar to that in Remark \ref{GN-REM}.
\end{proof}

\subsection{Generalized Compound Probability Theorem}
In this subsection we generalize the result obtained in Subsection \ref{SUB-COMPOUND}. Given the family\\
$\mathcal{F} = \{A_1|H, A_2|A_1H, \ldots, A_n|A_1 \cdots A_{n-1}H\}$,
by iteratively exploiting the associative property, we have
\[
\mathcal{C}(\mathcal{F}) = \mathcal{C}(\mathcal{C}(A_1|H, A_2|A_1H), A_3|A_1A_2H, \ldots, A_n|A_1 \cdots A_{n-1}H) = \]
\[ = \mathcal{C}(A_1A_2|H, A_3|A_1A_2H, \ldots, A_n|A_1 \cdots A_{n-1}H) = \cdots = A_1A_2 \cdots A_n|H \,;
\]
thus, by the compound probability theorem, if the assessment $\mathcal{P} = (p_1,\ldots,p_n,z)$ on
$\mathcal{F} \cup \{\mathcal{C}(\mathcal{F})\}$ is coherent, then
\[
z=l=u=p_1p_2 \cdots p_n =T_P(p_1,p_2,\ldots,p_n).
\]
\section{Further Aspects on Quasi Conjunction: from Bounds on Conclusions to Bounds on Premises in Quasi And rule}
\label{FURTH-QC} In this section, we study the propagation of probability bounds on the conclusion of the Quasi And rule to its premises. We start with the case of two premises $A|H$ and $B|K$, by examining probabilistic aspects on the lower and upper bounds, $l$ and $u$, for the probability of the conclusion $\mathcal{C}(A|H, B|K)$. More precisely, given any number $\gamma \in [0,1]$, we find: \\
(i) the set $\mathcal{L}_{\gamma}$ of the coherent
assessments $(x,y)$ on $\{A|H, B|K\}$ such that, for each $(x,y) \in \mathcal{L}_{\gamma}$, one has $l \geq \gamma$; \\
(ii) the set $\mathcal{U}_{\gamma}$ of the coherent
assessments $(x,y)$ on $\{A|H, B|K\}$ such that, for each $(x,y) \in \mathcal{U}_{\gamma}$, one has $u \leq \gamma$. \\
Case (i). Of course, $\mathcal{L}_0 = [0,1]^2$;
hence we can assume $\gamma > 0$. It must be $l =$ max$\{x+y-1,0\} \geq \gamma$, i.e.,  $x+y \geq
1+\gamma$ (as $\gamma > 0$); hence $\mathcal{L}_{\gamma}$ coincides with
the triangle having the vertices $(1,1),
(1,\gamma), (\gamma,1)$; that is
\[\mathcal{L}_{\gamma} = \{(x,y): \gamma \leq x \leq 1,\, 1 + \gamma - x \leq y \leq 1\} \,. \]
Notice that $\mathcal{L}_1 = \{(1,1)\}$; moreover, for $\gamma \in
(0,1)$, $(\gamma, \gamma) \notin \mathcal{L}_{\gamma}$. \\
Case (ii). Of course, $\mathcal{U}_1 = [0,1]^2$; hence we can assume $\gamma < 1$. We recall that $u= S_0^H(x,y)$, then in order the inequality  $S_0^H(x,y) \leq \gamma$ be satisfied, it must be $x < 1$ and $y < 1$. Thus, $u \leq \gamma$  if and only if $ \frac{x+y-2xy}{1-xy} \leq  \gamma $. Given any $x<1, y<1$, we have
\begin{equation}\label{XYU}
u-x=\frac{y(1-x)^2}{1-xy} \geq 0 \,,\;\; u-y=\frac{x(1-y)^2}{1-xy} \geq 0 \,;
\end{equation}
then, from $u \leq \gamma$ it follows $x \leq
\gamma, y \leq \gamma$; hence $\mathcal{U}_{\gamma} \subseteq [0,\gamma]^2$. Then,
taking into account that $x \leq \gamma$ and hence
\[
1-(2-\gamma)x = 1-2x+\gamma x \geq 1 - 2x + x^2 = (1-x)^2 > 0 \,,
\] we have
\begin{equation}\label{UPGAM}\frac{x+y-2xy}{1-xy} \leq \gamma \; \Longleftrightarrow \; y
\leq \frac{\gamma - x}{1-(2-\gamma)x} \,;\end{equation} therefore
\[
\mathcal{U}_{\gamma}= \left\{(x,y) : \; 0 \leq x \leq \gamma \,,\;\; y \leq
\frac{\gamma -
x}{1-(2-\gamma)x} \right\} \,. \] \\
Notice that $\mathcal{U}_0 = \{(0,0)\}$; moreover, for $x=y=\gamma \in
(0,1)$, it is $u=\frac{2\gamma}{1+\gamma} > \gamma$; hence, for
$\gamma \in (0,1)$, $\mathcal{U}_{\gamma}$ is a strict subset of
$[0,\gamma]^2$. \\ Of course, for every $(x,y) \notin
\mathcal{L}_{\gamma} \cup \mathcal{U}_{\gamma}$, it is $l < \gamma < u$.  Figure \ref{fig6} displays the sets $\mathcal{L}_{\gamma}, \mathcal{U}_{\gamma}$ when $\gamma=0.6$.
\begin{figure}[!ht]
\begin{center}
\includegraphics[ scale=\myscale]{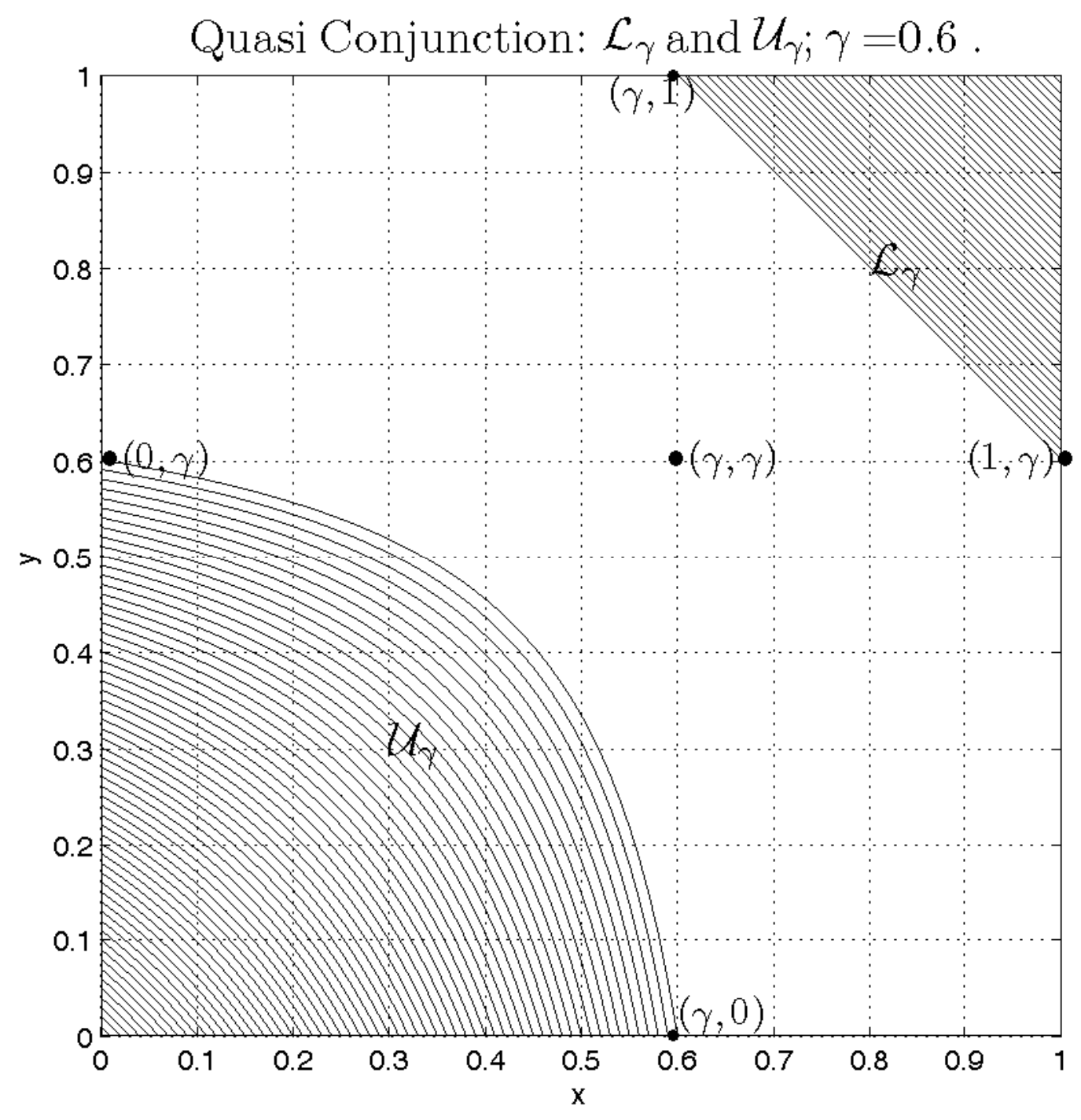}
\caption{The sets $\mathcal{L}_{\gamma}, \mathcal{U}_{\gamma}$.}
\label{fig6}
\end{center}
\end{figure}

In the next result we determine in general the sets $\mathcal{L}_{\gamma}, \mathcal{U}_{\gamma}$.
\begin{theorem}\label{LU-GEN} Given a coherent assessment $(p_1, p_2, \ldots,
p_n)$ on the family $\{E_1|H_1, \ldots, E_n|H_n\}$, where $E_1, H_1, \ldots, E_n, H_n$ are logically independent, we have
\begin{small}
\begin{equation}\label{LGAMMA}\begin{array}{l}
\mathcal{L}_{\gamma} = \{(p_1, \ldots, p_n) \in [0,1]^n : p_1 + \cdots +
p_n  \geq \gamma + n-1\} \,,\; \gamma > 0 \,,
\\ \\
\mathcal{U}_{\gamma} = \{(p_1, \ldots, p_n) \in [0,1]^n : 0 \leq p_1 \leq
\gamma \,,\;\; p_{k+1} \leq r_k
\,,\; k=1,\ldots,n-1\} \,,\; \gamma < 1 \,,
\end{array}
\end{equation}
\end{small}
where $r_k = \frac{\gamma - u_k}{1-(2-\gamma)u_k} \,,\, u_k = S_0^H(p_1,\ldots,p_{k})$, with $\mathcal{L}_0 = \mathcal{U}_1 = [0,1]^n$.
\end{theorem}
\begin{proof}
Of course, $\L_0 = [0,1]^n$, so that we can assume $\gamma >
0$. It must be
$l_n = \mbox{max}(p_1 + \cdots + p_n - (n-1), 0) \geq \gamma$, that
is, as $\gamma > 0$,
$p_1 + \cdots + p_n  \geq \gamma + n-1$.
Hence: $\L_{\gamma} = \{(p_1, \ldots, p_n) \in [0,1]^n : p_1 + \cdots +
p_n  \geq \gamma + n-1\}$. \\
We observe that $\L_{\gamma}$ is a convex polyhedron with vertices the points
\[
\begin{array}{ll}
V_1=(\gamma,1,\dots,1), \; V_2=(1, \gamma,1,\ldots,1), \; \cdots, \\
\; V_n=(1, \ldots,1,\gamma), \; V_{n+1}=(1,1,\ldots, 1) \,.
\end{array}
\]
Moreover, the convex hull of the vertices $V_1,\ldots,V_n$ is the subset of the points $(p_1, \ldots, p_n)$
of $\L_{\gamma}$ such that $l_n=\gamma$, that is such that $p_1 + \cdots +
p_n = \gamma+n-1$. \\
Now, let us determine the set $\U_{\gamma}$. Of course, $\U_1 = [0,1]^n$, so that we can assume $\gamma <
1$. We recall that $u_2, \ldots, u_n$ are the upper bounds on
$\C(\F_2), \ldots, \C(\F_n)$ associated with $(p_1,\ldots,p_n)$. Then,  from the relations
\[\C(\F_{k+1}) = \C(\C(\F_k), E_{k+1}|H_{k+1}) \,,\; k=2,\ldots,n-1 \,, \] by applying (\ref{XYU}) with $x=u_k, y=p_{k+1}$, we have that in order the inequality $u_{k+1} \leq \gamma$ be satisfied,it must be $u_k \leq \gamma, p_{k+1} \leq \gamma,\, k=2,\ldots,n-1$. Therefore
\[
u_n \leq \gamma \; \Longrightarrow \; p_1  \leq \gamma, \ldots, p_n  \leq \gamma, u_2  \leq \gamma, \ldots, u_{n-1}  \leq \gamma \,,
\]
so that $\U_{\gamma} \subseteq [0,\gamma]^n$. By iteratively
applying (\ref{UPGAM}), we obtain
\[
0 \leq p_1 \leq \gamma \,,\;\; p_2 \leq \frac{\gamma -
p_1}{1-(2-\gamma)p_1} \;\;\; \Longrightarrow \;\;\; u_2 \leq
\gamma \,,
\]
\[
0 \leq u_2 \leq \gamma \,,\;\; p_3 \leq \frac{\gamma -
u_2}{1-(2-\gamma)u_2} \;\;\; \Longrightarrow \;\; u_3 \leq \gamma
\,,
\]
\[
\vdots
\]
\[
0 \leq u_{n-1} \leq \gamma \,,\;\; p_n \leq \frac{\gamma -
u_{n-1}}{1-(2-\gamma)u_{n-1}} \;\;\; \Longrightarrow \;\; u_n \leq
\gamma \,.
\]\\
Therefore, observing that $u_1 = p_1$,
\begin{small}
\[
\U_{\gamma} = \{(p_1, \ldots, p_n) \in [0,1]^n : 0 \leq p_1 \leq
\gamma \,,\;\; p_{k+1} \leq \frac{\gamma - u_k}{1-(2-\gamma)u_k}
\,,\; k=1,\ldots,n-1\} \,.
\]
\end{small}
We observe that $\U_0 = \{(0, \ldots,0)\}$; moreover, for
$p_1=\cdots =p_n=\gamma \in (0,1)$, we obtain (by induction)
\[u_2=\frac{2\gamma}{1+\gamma} > \gamma \,,\;
u_3=\frac{3\gamma}{1+2\gamma} > \gamma \,,\; \cdots \,,\;
u_n=\frac{n\gamma}{1+(n-1)\gamma} > \gamma \,; \] hence, for
$\gamma \in (0,1)$, $\U_{\gamma}$ is a strict subset of
$[0,\gamma]^n$. \\
Of course, for every $(p_1,\ldots,p_n) \notin \L_{\gamma} \cup
\U_{\gamma}$, it is $l_n < \gamma < u_n$. As an example, for
$p_1=\cdots =p_n=\gamma \in (0,1)$, one has
\[
l_n = \mbox{max}(n\gamma - (n-1), 0) \;<\; \gamma \;<\; u_n =
\frac{n\gamma}{1+(n-1)\gamma} \,.
\]
\end{proof}
\section{Lower and Upper Bounds for Quasi Disjunction}
\label{SEC-QD}
We recall below the notion of quasi disjunction of conditional events  as defined in \cite{Adam75}.
\begin{definition}\label{QD}{\rm
Given any events $A, H, B, K$, with $H \neq
\emptyset, K \neq \emptyset$, the quasi disjunction of the
conditional events $A|H$ and $B|K$ is
the conditional event $\mathcal{D}(A|H,B|K) = (AH  \vee BK)|(H \vee K)$.
}\end{definition}
The constituents generated  by the family
$\{A|H, B|K, \D(A|H,B|K)\}$ and the corresponding points $Q_h$'s
are given in columns 2 and 8  of Table \ref{TABLE-QC}. In general, given a family of $n$ conditional events $\mathcal{F}_n=\{E_i|H_i,\, i=1,\ldots,n\}$, it is $\mathcal{D}(\mathcal{F}_n)=\mathcal{D}(E_1|H_1,\ldots,E_n|H_n) = (\bigvee_{i=1}^n E_iH_i)\big |(\bigvee_{i=1}^n H_i)$.
Quasi disjunction is associative; that is, for every subset $J \subset \{1,\ldots,n\}$, we have
$\mathcal{D}(\mathcal{F}_n)$ $=\mathcal{D}(\mathcal{F}_J \cup \mathcal{F}_{\Gamma})$ $= \mathcal{D}[\mathcal{D}(\mathcal{F}_J), \mathcal{D}(\mathcal{F}_{\Gamma})]$, where $\Gamma = \{1,\ldots,n\} \setminus J$.
\begin{remark}\label{REM-TS}
We recall that the quasi conjunction of $A|H$ and $B|K$ can also be written as $\C(A|H,B|K) = (A \lor H^c)\land(B \lor K^c)|(H \lor K)$; then, based on the usual negation operation $(E|H)^c = E^c|H$, it holds that
\begin{equation}\label{DUAL-QC-QD}\begin{array}{l}
[\C(A^c|H,B^c|K)]^c = [(A^c \lor H^c)\land(B^c \lor K^c)|(H \lor K)]^c = \\ = (AH \lor BK)|(H \lor K) = \D(A|H,B|K) \,,
\end{array}\end{equation}
which represents the De Morgan duality between quasi conjunction and quasi disjunction.
We also have
$\D(A|H,B|K) \, \vee \, \C(A^c|H,B^c|K) = \Omega|(H \vee K)$ and $\D(A|H,B|K) \, \wedge \, \C(A^c|H,B^c|K) = \emptyset|(H\vee K)$.
From (\ref{DUAL-QC-QD}) it follows
\begin{equation}\label{QD=1-QC}
P[\D(A|H,B|K)]=1-P[\C(A^c|H,B^c|K)] \,,
\end{equation}
which will be exploited in the next result.
\end{remark}
\begin{proposition}
\label{LOW-UP-QD-TWO}
{\rm
Assuming $A, H, B, K$ logically independent, we have: \\ (i) the probability assessment $(x,y)$ on $\{A|H, B|K\}$ is coherent for every $(x,y) \in [0,1]^2$; \\ (ii) given a coherent
assessment $(x,y)$ on $\{A|H, B|K\}$, the assessment $\mathcal{P} = (x,y,z)$ on
$\mathcal{F} = \{A|H, B|K, \mathcal{D}(A|H,B|K)\}$, with $z = P[\mathcal{D}(A|H,B|K)]$, is
a coherent extension of $(x,y)$ if and only  if  $z\in[l,u]$, where
\[
l=T_0^H(x,y)=\left\{\begin{array}{ll} \frac{xy}{x+y-xy}, & (x,y) \neq
(0,0), \\ 0, & (x,y)=(0,0),
\end{array} \right. \;\;\; u=S_L(x,y) = \min(x+y,1)\,.
\]
}\end{proposition}
\begin{proof} We observe that, by (\ref{EQ-LOW-UP-QC}), the extension $\gamma = P[\C(A^c|H,B^c|K)]$ of the assessment $P(A|H)=x, P(B|K)=y$ is coherent if and only if $\gamma' \leq \gamma \leq \gamma''$, where $\gamma'=T_L(1-x,1-y),\, \gamma''= S_0^H(1-x,1-y).$ Then, based on (\ref{QD=1-QC}) and on the results given in Appendix A, it follows that
\[
l=1-S_0^H(1-x,1-y)=T_0^H(x,y) \,,\;\; u=1-T_L(1-x,1-y)=S_L(x,y) \,.
\]
\end{proof}
In Figure \ref{fig2} is shown the convex hull $\I$ for given values $x,y$, with the associated interval $[l,u]$ of coherent extensions $z=P[\D(A|H,B|K)]$. As for quasi conjunction, the convex hull $\I$ does not depend on $z$. In the next subsections we examine some particular cases.
\begin{figure}[!ht]
\begin{center}
\includegraphics[ scale=\myscale]{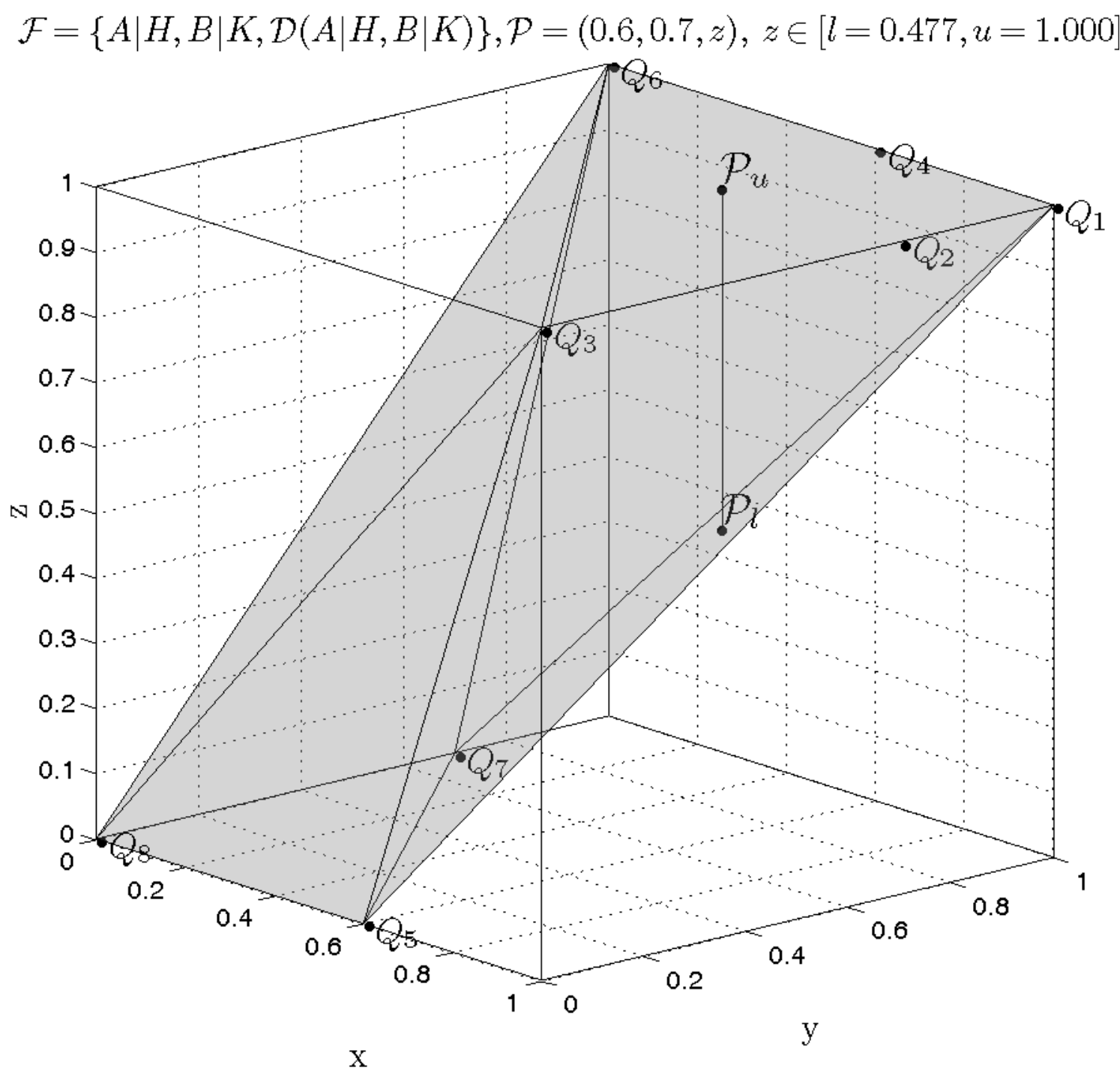}
\caption{The convex hull $\I$  associated with the pair $(\F,\P)$ in case of quasi disjunction without logical relations.
The  interval $[l,u]$ for $z=P[\D(A|H,B|K)]$ is the range of the third coordinate $z$ of each $\P\in \overline{\P_l\P_u}=\{(x,y,z):z\in[T_0^H(x,y),S_{L}(x,y)]\}$.
The segment $\overline{\P_l\P_u}$  is the intersection  between the segment $\{(x,y,z):z\in [0,1]\}$  and the convex hull $\I$.}
\label{fig2}
\end{center}
\end{figure}
\subsection{The Dual of  Compound Probability Theorem}
Given any logically independent events $A,\,B,\, H$, with $A^cH \neq \emptyset$, the assessment $(x,y)$ on $\{A|H, B|A^cH \}$ is coherent, for every $(x,y)\in[0,1]^2$.  We have $\mathcal{D}(A|H,B|A^cH) = (A\vee B)|H$ and, defining $z = P(A\vee B|H)$, by (\ref{QD=1-QC}) and by the results in Subsection \ref{SUB-COMPOUND},  we have
\[
\begin{array}{ll}
z=P(\mathcal{D}(A|H,B|A^cH))=1-P(\mathcal{C}(A^c|H,B^c|A^cH))=\\=
1-T_P(1-x,1-y)=x+y-xy=S_P(x,y) \,;
\end{array}
\]
that is $z$ is equal to the {\em probabilistic sum} of $x,y$.

\subsection{The case $A|H\subseteq B|K$}
\label{GNQD-SUB}
From $A|H \subseteq B|K$ we have $\mathcal{D}(A|H,B|K)=(BK) \,|\, (H \vee K)$. Then,
as shown by Table \ref{TABLE-QC}  and by Remark \ref{GN-REM}), it holds that
\[
t(A|H)\leq t(\mathcal{C}(A|H,B|K) )\leq t(\mathcal{D}(A|H,B|K) )\leq  t(B|K) \,.
\]
Then: $A|H \subseteq B|K$ implies $A|H \subseteq \mathcal{C}(A|H,B|K) \subseteq \mathcal{D}(A|H,B|K) \subseteq B|K$.
We recall that, by Remark \ref{REM-GN}, $A|H \subseteq B|K$ amounts to $B^c|K \subseteq A^c|H$; then, given the assessment $P(A|H)=x, P(B|K)=y$, where $x \leq y$, by applying Proposition~\ref{PROP-QC-GN} to the family $\{B^c|K, A^c|H\}$, the extension $\gamma = P[\C(B^c|K, A^c|H)]$ of $(x,y)$ is coherent if and only if $\gamma' \leq \gamma \leq \gamma''$, where $\gamma'=1-y,\, \gamma''= 1-x$. Then, by (\ref{QD=1-QC}), the extension $z = P[\mathcal{D}(A|H,B|K)]$ of $(x, y)$ is coherent if and only if $l \leq z \leq u$, where $l = x = \min(x,y) \,,\; u = y = \max(x,y)$.
\subsection{Quasi Conjunction, Quasi Disjunction and Or Rule.}
We recall that in Or rule with premises $H\normally A$ and $K\normally A$ the conclusion is  $H\vee K\normally A$.  Moreover, for the conditional events $A|H$ and $A|K$ associated with the premises, we have
\[
\C(A|H,A|K) = \D(A|H,A|K)= A|(H \vee K)  \,,
\]
which is the conditional event associated with the conclusion $H\vee K\normally A$ of Or rule.
In \cite{Gili02} it has been proved that, under logical independence of $A, H, K$, the assessment $z=P(A|(H \vee K)$ is a coherent extension of the assessment $(x,y)$ on $\{A|H, A|K\}$ if and only if $z\in [l,u]$, with
\begin{equation}\label{EQ-LOW-UP-OR}
l=T_0^H(x,y) \,,\;\;\;  u=S_0^H(x,y) \,.
\end{equation}
The convex hull $\I$ for  given values $x,y$ and the associated interval $[l,u]$ for $z=P[\D(A|H,A|K)]$ are shown in Figure \ref{fig8}.
\begin{figure}[!ht]
\begin{center}
\includegraphics[ scale=\myscale]{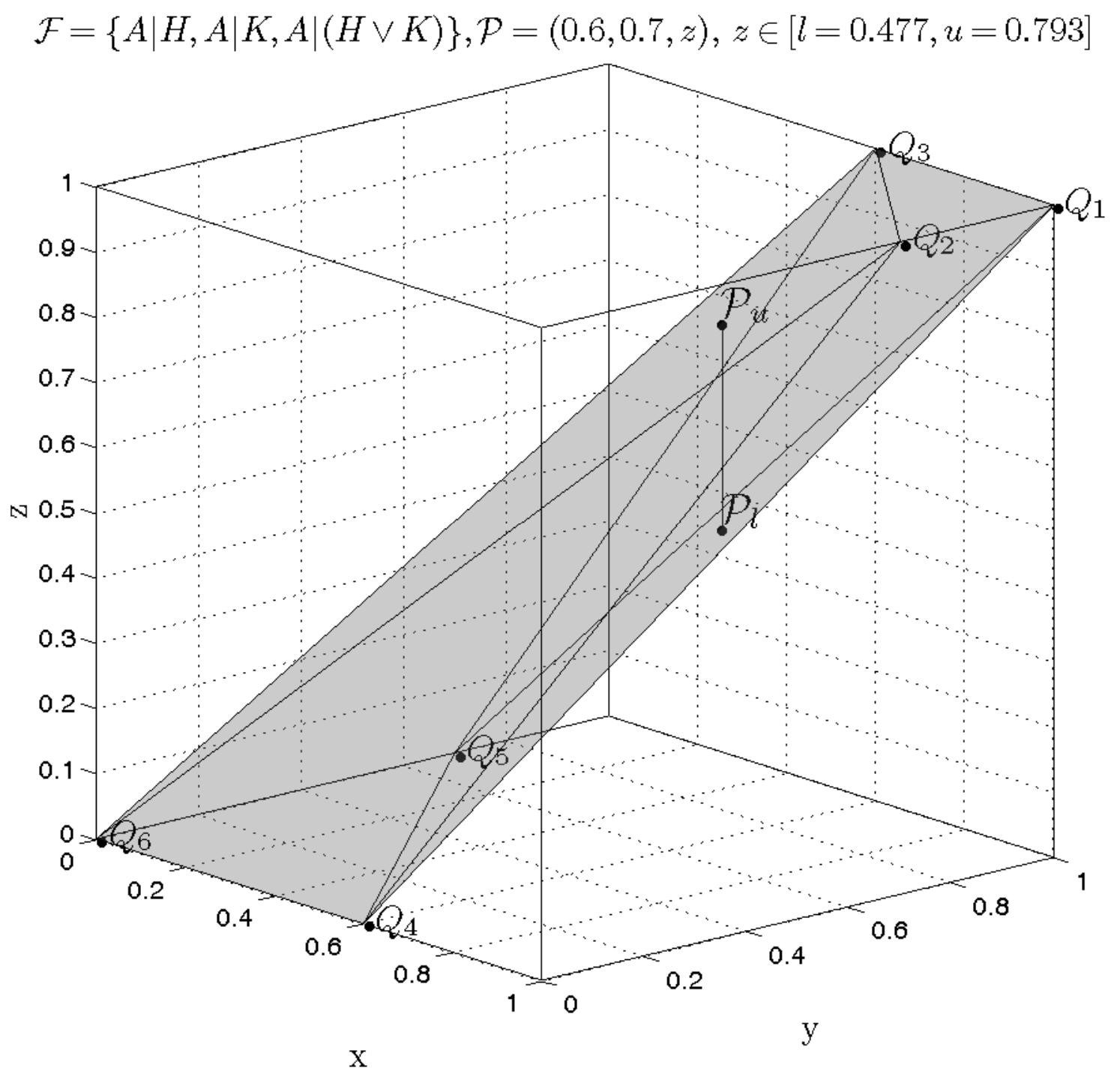}
\caption{Convex hull $\I$  associated with the pair $(\F,\P)$ for  the  Or rule. The  interval $[l,u]$ for $z=P[\D(A|H,A|K)]=P[\C(A|H,A|K)]$ is the range of the third coordinate $z$ of each $\P\in\overline{\P_l\P_u}=\{(x,y,z):z\in[T_0^H(x,y),S_{0}^{H}(x,y)]\}$. The segment $\overline{\P_l\P_u}$  is the intersection  between the segment $\{(x,y,z):z\in [0,1]\}$  and the convex hull $\I$.}
\label{fig8}
\end{center}
\end{figure}
\subsection{Lower and Upper Bounds for the Quasi Disjunction of $n$ Conditional Events}
 Given the family $\mathcal{F}_n = \{E_1|H_1, \ldots, E_n|H_n\}$, let us consider
the quasi disjunction $\mathcal{D}(\mathcal{F}_n)$ of the conditional events in
$\mathcal{F}_n$. By the associative property of quasi disjunction,
defining $\mathcal{F}_k = \{E_1|H_1, \ldots, E_k|H_k\}$, for each $k=2,
\ldots,n$ it holds that
 $\mathcal{D}(\mathcal{F}_k) = \mathcal{D}(\mathcal{D}(\mathcal{F}_{k-1}), E_k|H_k)$. Then, denoting by $T_0^H$  the Hamacher t-norm with parameter $\lambda=0$ and by $S_L$ the Lukasiewicz t-conorm (see \ref{SEZ-TNORMN} and \ref{HAM-NORM-CON}), we have
\begin{theorem}\label{LOW-UP-QD} {\rm Given a probability assessment
$\mathcal{P}_n = (p_1, p_2, \ldots, p_n)$ on $\mathcal{F}_n = \{E_1|H_1, \ldots,
E_n|H_n\}$, let $[l_k,u_k]$ be
the interval of coherent extensions of the assessment $\mathcal{P}_k = (p_1, p_2, \ldots, p_k)$ on
the quasi disjunction $\mathcal{D}(\mathcal{F}_k)$, where $\mathcal{F}_k = \{E_1|H_1, \ldots,
E_k|H_k\}$. Then, assuming $E_1, H_1, \ldots, E_n, H_n$ logically independent,
for each $k=2,\ldots,n$, we have
\[
l_k = T_0^H(p_1,p_2, \ldots, p_k) \,,\;\;\; u_k = S_L(p_1,p_2, \ldots,
p_k) \,.
\]
}
\end{theorem}
\begin{proof}
Of course, from Proposition \ref{LOW-UP-QD-TWO} it is $l_2=T_0^H(p_1,p_2)$ and $u_2=S_L(p_1,p_2)$.  The rest of the proof is similar to that one in Theorem \ref{LOW-UP-QC}.
\end{proof}
\begin{remark}\label{REM-C-IMPLIES-D}
Given any conditional events $A|H$ and $B|K$, as shown in
Table \ref{TABLE-QC}, it holds that
$t(\C(A|H,B|K))\leq t(\D(A|H,B|K))$, which amounts to $\C(A|H,B|K))\subseteq \D(A|H,B|K)$.
In  general, given a finite  family of conditional events $\F_n$, we have $t(\C(\F_n))\leq t(\D(\F_n))$, that is $\C(\F_n) \subseteq \D(\F_n)$, so that $P[\C(\F_n)] \leq P[\D(\F_n)]$. Thus, if the family $\mathcal{F}_n$ is p-consistent, then $\F_n \Rightarrow_p \C(\F_n)\Rightarrow_p \D(\F_n)$ and we obtain the following \textit{Quasi Or} rule
\begin{equation}
\begin{array}{llll}
\mbox{(Quasi Or)} &  \hspace{2.5cm} & \F_n \Rightarrow_p  \D(\F_n) \;. &\hspace{3.5cm}
\end{array}
\end{equation}
We observe that Quasi Or rule also follows directly from Theorem \ref{LOW-UP-QD}.
\end{remark}
\subsection{General Or Rule}
Let us consider the general Or rule (see \cite{Gili12}), where the premises are the conditional events $E|H_1, \ldots, E|H_n$ and the conclusion is the conditional event $E|(H_1\vee H_2 \vee \ldots, \vee H_n)$. By the associative property of quasi disjunction,
defining $\mathcal{F}_k = \{E|H_1, \ldots, E|H_k\}$, for each $k=2,
\ldots,n$ it holds that
 \[
 \mathcal{D}(\mathcal{F}_k) = \mathcal{D}(\mathcal{D}(\mathcal{F}_{k-1}), E_k|H_k)=E|(H_1\vee H_2 \vee \ldots, \vee H_k) \,.
 \]
 We also observe that  $\mathcal{D}(\mathcal{F}_k)=\mathcal{C}(\mathcal{F}_k)$.
 Then,  by exploiting the notions of t-norm, t-conorm, quasi disjunction and quasi conjunction, Theorem 9 in \cite{Gili12} can be written as
\begin{theorem}\label{LOW-UP-OR} {\rm Given a probability assessment
$\mathcal{P}_n = (p_1, p_2, \ldots, p_n)$ on $\mathcal{F}_n = \{E|H_1, E|H_2, \ldots,
E|H_n\}$, let $[l_k,u_k]$ be
the interval of coherent extensions of the assessment $\mathcal{P}_k = (p_1, p_2, \ldots, p_k)$ on
the quasi disjunction $\mathcal{D}(\mathcal{F}_k)$, where $\mathcal{F}_k = \{E|H_1, \ldots,
E|H_k\}$. Then, assuming $E, H_1, \ldots, H_n$ logically independent,
for each $k=2,\ldots,n$, we have
\[
l_k = T_0^H(p_1,p_2, \ldots, p_k) \,,\;\;\;
u_k = S_0^H(p_1,p_2, \ldots, p_k) \,.
\]}
\end{theorem}
\begin{proof}
 Of course, from (\ref{EQ-LOW-UP-OR}) it is $l_2=T_0^H(p_1,p_2)$ and $u_2=S_0^H(p_1,p_2)$.  The rest of the proof is similar to that one in Theorem \ref{LOW-UP-QC}.
\end{proof}
In \cite[Theorem 9]{Gili12}), by implicitly assuming $(p_1,\ldots,p_k) \in (0,1)^k$, it has been proved by a direct probabilistic analysis that \[
l_k = \frac{1}{1 + \sum_{i=1}^k \frac{1-p_i}{p_i}} \,,\;\;\;
u_k = \frac{\sum_{i=1}^k \frac{p_i}{1-p_i}}{1 + \sum_{i=1}^k
\frac{p_i}{1-p_i}} \,.
\]
By adopting the conventions $\frac{1}{\infty} = 0, \frac{1}{0}=\infty\,, \frac{\infty}{\infty} = 1$, the previous formulas hold in general for every $(p_1,\ldots,p_k) \in [0,1]^k$. In  \ref{HAM-NORM-CON} the previous expressions for the Hamacher t-norm and t-conorm have been derived by using the notion of additive generator.
\begin{example} An application of Or rule is obtained by imagining a medical scenario with a disease $E$ and $n$ symptoms $H_1, \ldots, H_n$, with $P(E|H_i)=p_i,\, i=1,\ldots,n$, and $P(E|(H_1 \vee \cdots \vee H_n) \in [l_n,u_n]$. If, for instance, $p_1=\cdots=p_n=1-\varepsilon$, from Theorem \ref{LOW-UP-OR} it follows $l_n = T_0^H(1-\varepsilon, \ldots, 1-\varepsilon) = \frac{1-\varepsilon}{1 + (n-1)\varepsilon}$ and $u_n = S_0^H(1-\varepsilon, \ldots, 1-\varepsilon) = \frac{n(1-\varepsilon)}{\varepsilon + n(1-\varepsilon)}$. Then: (i) for $\varepsilon \rightarrow 0$, we have $l_n \rightarrow 1$ and $u_n \rightarrow 1$; (ii) for $n \rightarrow +\infty$ we have $l_n \rightarrow 0$ and $u_n \rightarrow 1$. As we can see, in the second case the interval $[l_n,u_n]$ gets wider and wider as the number of premises increases. An interesting related phenomenon where additional information leads to less informative conclusion is the pseudodiagnosticity task, studied in the psychology of
uncertain reasoning (\cite{Klei13,TwDK10}).
\end{example}
\section{Further Aspects on Quasi Disjunction: from Bounds on Conclusions to Bounds on Premises in Quasi Or rule}
In this section, we study the propagation of probability bounds on the conclusion of the Quasi Or rule to its premises. We start with the case of two premises $A|H$ and $B|K$, by examining probabilistic aspects on the lower and upper bounds, $l$ and $u$, for the probability of the conclusion $\mathcal{D}(A|H, B|K)$. More precisely, given any number $\gamma \in [0,1]$, we find:\\
$(i)$ the set $\mathbf{L}_{\gamma}$ of the coherent assessments $(x,y)$ on $\{A|H, B|K\}$ such that, for each $(x,y) \in \mathbf{L}_{\gamma}$, one has $l \geq \gamma$; \\
$(ii)$ the set $\mathbf{U}_{\gamma}$ of the coherent assessments $(x,y)$ on $\{A|H, B|K\}$ such that, for each $(x,y) \in \mathbf{U}_{\gamma}$, one has $u \leq \gamma$. \\
Case $(i)$.
 Let be given $\gamma \in [0,1]$. We denote by
$\mathbf{L}_{\gamma}$ the set of coherent assessments $(x,y)$ on $\{A|H,
B|K\}$ which imply $z \geq \gamma$. Of course, $\mathbf{L}_0 = [0,1]^2$;
hence we can assume $\gamma > 0$. We recall that $l= T_0^H(x,y)$, then in order the inequality  $T_0^H(x,y)\geq \gamma$ be satisfied, it must be $x > 0$ and $y > 0$. Thus,   $l \geq \gamma$  if and only if $\frac{xy}{x+y-xy}\geq \gamma$.
We have
\begin{equation} \label{XYL}
x-l=\frac{x^2(1-y)}{x+y-xy}\geq 0 \,, \;\;\; y-l=\frac{y^2(1-x)}{x+y-xy}\geq 0 \,;
\end{equation}
then, from $l \geq \gamma$ it follows $x \geq \gamma$, $y \geq \gamma$; thus $\mathbf{L}_{\gamma}\subseteq [\gamma,1]^2$. Then,
taking into account that $x \geq \gamma$ and hence $x(1+\gamma)-\gamma> 0$, we have
\begin{equation}\label{UPGAM-QD}
\frac{xy}{x+y-xy}\geq \gamma\; \Longleftrightarrow \; y
\geq \frac{\gamma x}{x(1+\gamma)-\gamma} \,;\end{equation} therefore
\[
\mathbf{L}_{\gamma}= \left\{(x,y) : \;  \gamma \leq x \leq 1 \,,\;\; y \geq
\frac{\gamma x}{x-\gamma+\gamma x}
 \right\} \,. \] \\
Notice that $\mathbf{L}_1 = \{(1,1)\}$;
 for $x=y=\gamma \in
(0,1)$, it is $l=\frac{\gamma}{2-\gamma} < \gamma$; hence, for $\gamma \in (0,1)$, $\mathbf{L}_{\gamma}$ is a strict subset of $[\gamma, 1 ]^2$.
\\
Case $(ii)$. Of course, $\mathbf{U}_1 = [0,1]^2$;
hence we can assume $\gamma < 1$. It must be $u =$ min$\{x+y,1\} \leq \gamma$, i.e.,  $x+y \leq
\gamma$ (as $\gamma < 1$); hence $\mathbf{U}_{\gamma}$ coincides with
the triangle having the vertices $(0,0),
(0,\gamma), (\gamma,0)$; that is
\[\mathbf{U}_{\gamma} = \{(x,y): 0 \leq x \leq \gamma,\, 0 \leq y \leq \gamma - x\} \,. \]
Notice that $\mathbf{U}_0 = \{(0,0)\}$; moreover, for $\gamma \in
(0,1)$, $(\gamma, \gamma) \notin \mathbf{U}_{\gamma}$. \\
 Of course, for every $(x,y) \notin
\mathbf{L}_{\gamma} \cup \mathbf{U}_{\gamma}$, it is $l < \gamma < u$. \\
{Figure \ref{fig7} displays the sets $\mathbf{L}_{\gamma}, \mathbf{U}_{\gamma}$ when $\gamma=0.4$.}
 \begin{figure}[!ht]
 \begin{center}
 \includegraphics[ scale=\myscale]{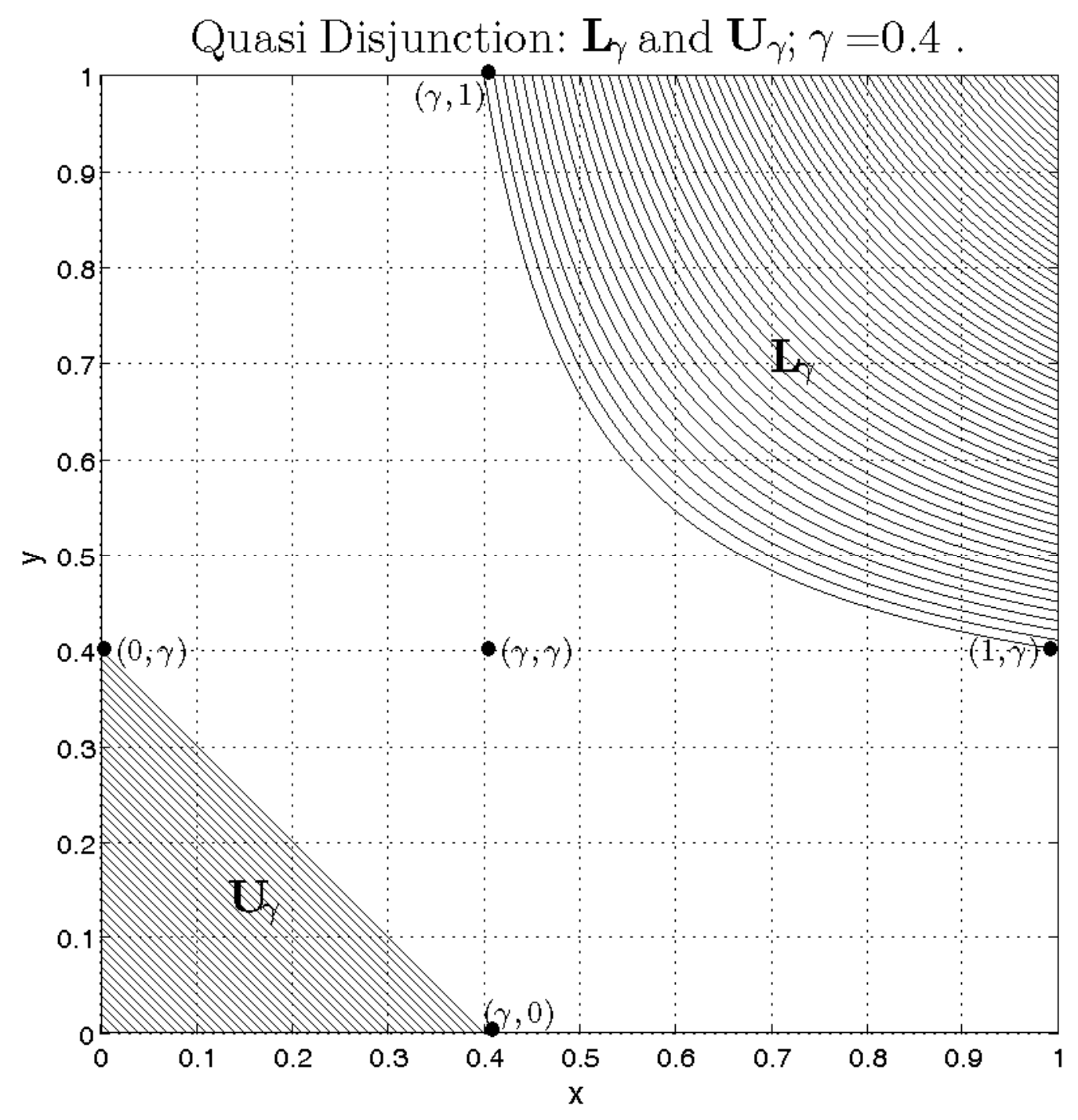}
 \caption{The sets $\mathbf{L}_{\gamma}, \mathbf{U}_{\gamma}$\,.}
 \label{fig7}
 \end{center}
 \end{figure}
In the next result we determine in general the sets $\mathbf{L}_{\gamma}, \mathbf{U}_{\gamma}$.
\begin{theorem}\label{LU-GEN-QD}
 Let be given the family $\F_n=\{E_1|H_1, \ldots, E_n|H_n\}$, with the events $E_1, H_1, \ldots, E_n, H_n$ logically independent. Moreover, for any given $\gamma \in [0,1]$ let  $\mathbf{L}_{\gamma}$  (\emph{resp.} $\mathbf{U}_{\gamma}$) be the set of the coherent assessments $(p_1, p_2, \ldots,
p_n)$ on $\F_n$ such that, for each $(p_1, p_2, \ldots, p_n) \in \mathbf{L}_{\gamma}$ (\emph{resp.} $(p_1,p_2, \ldots, p_n) \in \mathbf{U}_{\gamma}$), one has  $l \geq \gamma$ (\emph{resp.} $u \leq \gamma$), where $l$ is the lower  bound  (\emph{resp}. $u$ is the upper  bound) of the coherent extensions $z=P[\mathcal{D}(\F_n)]$. We have
\begin{small}
\begin{equation}\label{LGAMMA-QD}\begin{array}{l}
\mathbf{U}_{\gamma} = \{(p_1, \ldots, p_n) \in [0,1]^n : p_1 + \cdots +
p_n  \leq \gamma \} \,,\; \gamma < 1 \,,
\\ \\
\mathbf{L}_{\gamma} = \{(p_1, \ldots, p_n) \in [0,1]^n : \gamma \leq p_1 \leq
1 \,,\;\;  r_k \leq p_{k+1}
\,,\; k=1,\ldots,n-1\} \,,\; \gamma > 0 \,,
\end{array}
\end{equation}
\end{small}
where $r_k = \frac{\gamma l_k}{l_k-\gamma+ \gamma l_k} \,,\, l_k = T_0^H(p_1,\ldots,p_{k})$, with $\mathbf{L}_0 = \mathbf{U}_1 = [0,1]^n$.
\end{theorem}
\begin{proof}
Of course, $\mathbf{U}_1 = [0,1]^n$, so that we can assume $\gamma <1$. It must be
$u_n = \mbox{min}(p_1 + \cdots + p_n , 1) \leq \gamma$, that
is, as $\gamma < 1$,
$p_1 + \cdots + p_n  \leq \gamma $.
Hence: $\mathbf{U}_{\gamma} = \{(p_1, \ldots, p_n) \in [0,1]^n : p_1 + \cdots +
p_n  \leq \gamma \}$. \\
We observe that $\mathbf{U}_{\gamma}$ is a convex polyhedron with vertices the points
\[
\begin{array}{l}
V_1=(\gamma,0,\dots,0), \; V_2=(0, \gamma,0,\ldots,0), \; \cdots,\\
\; V_n=(0, \ldots,0,\gamma), \; V_{n+1}=(0,0,\ldots, 0) \,.
\end{array}
\]
Moreover, the convex hull of the vertices $V_1,\ldots,V_n$ is the subset of the points $(p_1, \ldots, p_n)$
of $\mathbf{U}_{\gamma}$ such that $u_n=\gamma$, that is such that $p_1 + \cdots +
p_n = \gamma$. \\

 Of course, $\mathbf{L}_0 = [0,1]^n$, so that we can assume $\gamma>0$. We recall that $l_2, \ldots, l_n$ are the lower bounds on
$\D(\F_2), \ldots, \D(\F_n)$ associated with $(p_1,\ldots,p_n)$. Then,  from the relations
\[\D(\F_{k+1}) = \D(\D(\F_k), E_{k+1}|H_{k+1}) \,,\; k=2,\ldots,n-1 \,, \] by applying (\ref{XYL} ) with $x=l_k, y=p_{k+1}$, we have that in order the inequality $l_{k+1} \geq \gamma$ be satisfied, it must be $l_k \geq \gamma, p_{k+1} \geq \gamma,\, k=2,\ldots,n-1$. Therefore
\[
l_n \geq \gamma \; \Longrightarrow \; p_1  \geq \gamma, \ldots, p_n  \geq \gamma, l_2  \geq \gamma, \ldots, l_{n-1}  \geq \gamma \,,
\]
so that $\mathbf{L}_{\gamma} \subseteq [\gamma,1]^n$. By iteratively
applying (\ref{UPGAM-QD}), we obtain

\[
\gamma \leq p_1 \leq 1 \,,\;\;
p_2
\geq \frac{\gamma p_1}{p_1(1+\gamma)-\gamma}
\;\;\; \Longrightarrow \;\;\; l_2 \geq
\gamma \,,
\]
\[
\gamma \leq l_2 \leq 1 \,,\;\;
p_3
\geq \frac{\gamma l_2}{l_2(1+\gamma)-\gamma}
\;\;\; \Longrightarrow \;\;\; l_3 \geq
\gamma \,,
\]
\[
\vdots
\]
\[
\gamma \leq l_{n-1} \leq 1 \,,\;\; p_n \geq \frac{\gamma l_{n-1}}{l_{n-1}(1+\gamma)-\gamma} \;\;\; \Longrightarrow \;\;\; l_n \geq
\gamma .
\]\\
Therefore,  observing that  $l_1 = p_1$, we have,
\[
\mathbf{L}_{\gamma} = \{(p_1, \ldots, p_n) \in [0,1]^n : \gamma \leq p_1 \leq 1 \,,\;\; p_{k+1} \geq \frac{\gamma l_k}{l_k(1+\gamma)-\gamma}
\,,\; k=1,\ldots,n-1\} \,.
\] \\
We observe that $\mathbf{L}_1 = \{(1, \ldots,1)\}$; moreover, for
$p_1=\cdots =p_n=\gamma \in (0,1)$, we obtain (by induction)
\[l_2=\frac{\gamma}{2-\gamma} < \gamma \,,\;
l_3=\frac{\gamma}{3-2\gamma} < \gamma \,,\; l_4=\frac{\gamma}{4-3\gamma}\cdots \,,\;
l_n=\frac{\gamma}{n-(n-1) \gamma} < \gamma \,; \] hence, for
$\gamma \in (0,1)$, $\mathbf{L}_{\gamma}$ is a strict subset of
$[0,\gamma]^n$.
\end{proof}
\section{Biconditional Events, $n$-Conditional Events and Loop Rule}
\label{SUB-BICO}
We now examine the quasi conjunction of $A|B$ and $ B|A$, with  $A,\,B$ logically independent events.
We have
\[
\mathcal{C}(A|B,B|A) = (AB \vee B^c) \wedge (BA \vee A^c)\,|\,(A \vee B) =
AB \,|\, (A \vee B) \,.
\]
We observe that the conditional event $AB \,|\, (A \vee B)$ captures the notion of biconditional event\footnote{The representation of a biconditional event as a quasi conjunction was noticed in a private communication between A. Fugard and A. Gilio (January 2010).} $A\dashv\vdash B$ considered by some authors  as  the ``conjunction'' between $A|B$ and $B|A$ and has the same truth table of the ``defective biconditional''  discussed in \cite{GaBa09}; see also \cite{FPMK11}.
It can be easily verified  that, for every pair  $(x,y)\in [0,1]\times[0,1]$ the probability assessment $(x,y)$ on $\{A|B, B|A \}$
is coherent. Given any coherent assessment $(x,y)$ on $\{A|B,B|A\}$, the probability assessment $z = P(A\dashv\vdash B)$, is
a coherent extension of $(x,y)$ if and only if
\[
z=\left\{\begin{array}{ll}
0 & (x,y)=(0,0),\\
\frac{xy}{x+y-xy} & (x,y)\ne (0,0).
\end{array}\right.
\]
We can study the coherence of the assessment $\P=(x,y,z)$ on the family \[\mathcal{F} = \{A|B, B|A, A\dashv\vdash B\}=  \{A|B, B|A, AB \,|\, (A \vee B)\}\,,\]
by the geometrical approach described in Section \ref{PRELIM}. In such a case, as the events of the family are not logically independent, the constituents generated by the family and contained in $A\vee B$ are: $C_1 = AB,\, C_2 = AB^c,\, C_3 = A^cB$. We distinguish two cases: (i) $(x,y)\neq (0,0)$; (ii) $(x,y)= (0,0)$. \\
(i) If $(x,y)\neq (0,0)$ the corresponding points $Q_h$'s are
$Q_1 = (1,1,1),\, Q_2 = (x,0,0),\,  Q_3 = (0,y,0)$,
and, in our case, the coherence of $\P$ simply amounts to the
geometrical condition $\P \in \I$, where $\I$ is the triangle with vertices  $Q_1,Q_2,Q_3$.
Based on the equation of the plane containing  $\I$, we have that $\P$ is coherent if and only if:
$z = \frac{xy}{x+y-xy}$. \\
(ii) If $(x,y)=(0,0)$, then $Q_2=Q_3=(0,0,0)$ and the convex hull $\I$ is the segment $Q_1Q_2$. Then, $\P=(0,0,z)$ is coherent if and only if $z=0$.\\ Then,  the value $z$ is a coherent extension of $(x,y)$ if and only if
\[
z=T_0^H(x,y)=\left\{\begin{array}{ll}
0 & (x,y)=(0,0),\\
\frac{xy}{x+y-xy} & (x,y)\ne (0,0)\,,
\end{array}\right.
\]
where  $T_0^H(x,y)$ is  the Hamacher t-norm, with parameter $\lambda=0$, defined by formula (\ref{HAM-T0}).
In agreement with Remark \ref{REM-SC}, we observe that
\[
 T_L(x,y) \leq T_0^H(x,y) \leq S_0^H(x,y).
\]
\subsection{Generalizing Biconditional Events: An Application to Loop rule}
As  shown before, given any (non impossible) events $A_1, A_2$, the biconditional event associated with them is given by
\[
A_1\dashv\vdash A_2 = \C(A_2|A_1, A_1|A_2) = A_1A_2 \,|\,(A_1 \vee A_2) \,.
\]
The notion of biconditional event can be generalized by defining the \linebreak $n$-conditional event associated with $n$ (non impossible) events $A_1, \ldots, A_n$ as
\[
A_1\dashv\vdash A_2\dashv\vdash \cdots \dashv\vdash A_n = \C(A_2|A_1, \ldots, A_n|A_{n-1}, A_1|A_n) \,.
\]
Let $C_0,C_1,\ldots,C_m$ be the constituents generated by the conditional events $A_2|A_1, \ldots, A_n|A_{n-1}, A_1|A_n$. We set $C_0 = A_1^cA_2^c \cdots A_n^c$ and $C_1 = A_1A_2 \cdots A_n$; then, for each $h = 2,\ldots,m$, it is $C_h = A_{i_1} \cdots A_{i_r}A_{i_{r+1}}^c \cdots A_{i_n}^c$, with $1 \leq r<n$. As it can be easily verified, the truth value of the $n$-conditional associated with $C_h$ is true, or false, or void, according to whether $h=1$, or $h>1$, or $h=0$; then it holds that
\[
\C(A_2|A_1, \ldots, A_n|A_{n-1}, A_1|A_n) = A_1 \cdots A_n \,|\,(A_1 \vee \cdots \vee A_n) \,.
\]
In (\cite{Gili04}), where also the relationship with conditional objects (\cite{DuPr94}) has been studied, the previous formula has been obtained by a suitable inductive reasoning, by showing that: \\
(i) $\C(A_2|A_1, \ldots, A_n|A_{n-1}) = (E_0 \vee \cdots \vee E_{n-1})|(A_1 \vee \cdots \vee A_{n-1})$, \\
where $E_0 = A_1 \cdots A_n \,,\, E_1 = A_1^cA_2 \cdots A_n \,,\, \ldots \,,\, E_{n-2} = A_1^c \cdots A_{n-2}^cA_{n-1}A_n \,$, $E_{n-1} = A_1^c \cdots A_{n-1}^c$; \\
(ii) then
\[
\C(A_2|A_1, \ldots, A_n|A_{n-1}, A_1|A_n) = \C[(E_1 \vee \cdots \vee E_n)|(A_1 \vee \cdots \vee A_{n-1}), A_1|A_n] =
\]
\begin{equation}\label{QC-LOOP-N}
= A_1 \cdots A_n|(A_1 \vee \cdots \vee A_n) \,.
\end{equation}
Of course, for any given derangement (a permutation with no fixed point) $(i_1,i_2,\ldots,i_n)$ of $(1,2,\ldots,n)$, we have
\[
\C(A_{i_1}|A_1, \ldots, A_{i_{n-1}}|A_{n-1}, A_{i_n}|A_n)=\C(A_2|A_1, \ldots, A_n|A_{n-1}, A_1|A_n) \,;
\]
that is, the $n$-conditional $A_1\dashv\vdash A_2\dashv\vdash \cdots \dashv\vdash A_n$ can be represented as the quasi conjunction of the conditional events $A_2|A_1, \ldots, A_n|A_{n-1}, A_1|A_n$, or equivalently as the quasi conjunction of the conditional events $A_{i_1}|A_1, \ldots,$ $A_{i_{n-1}}|A_{n-1}, A_{i_n}|A_n$.
In particular for $(i_1,i_2,\ldots,i_n)=(n,1,2,\ldots,n-1)$ we have
\begin{equation}\label{QC-LOOP}
\C(A_1|A_2, \ldots, A_{n-1}|A_n, A_n|A_1) = \C(A_2|A_1, \ldots, A_n|A_{n-1}, A_1|A_n) \,.
\end{equation}
As a consequence, we can immediately obtain the
probabilistic interpretation of {\em Loop} rule (\cite{KrLM90}). Given $n$ logically
independent events $A_1, \, A_2, \; \ldots, \, A_n$, Loop rule is
the following one:
\begin{equation}\label{LOOP-1}
A_1 \; \s~ \, A_2 \, , \; A_2 \; \s~ \, A_3 \, , \; \cdots \, , \;
A_n \; \s~ \, A_1  \;\; \Longrightarrow \;\; A_1 \; \s~ \, A_n \,.
\end{equation}
In \cite{KrLM90} it has also been proved that, for every $i,j = 1,2,\ldots,n$, it holds that
\begin{equation}\label{LOOP-2}
A_1 \; \s~ \, A_2 \,, \; A_2 \; \s~ \, A_3 \,, \; \cdots \,, \;
A_n \; \s~ \, A_1  \;\; \Longrightarrow \;\; A_i \; \s~ \, A_j
\,.
\end{equation}
\subsection{Probabilistic Aspects on Loop Rule}
In our probabilistic approach, formula (\ref{LOOP-2}), which generalizes formula  (\ref{LOOP-1}), can be obtained by the following steps: \\
- given any p-consistent family of conditional events $\F$, from Corollary \ref{ENT-QC} it holds that $\F$ p-entails $\C(\F)$; \\
-  defining $\F = \{A_2|A_1, \ldots, A_{n}|A_{n-1}, A_1|A_n\}$, it can be checked that  $\F$ is p-consistent; then, for every $i,j = 1,2,\ldots,n$, by $(\ref{QC-LOOP-N})$ $\C(\F) \subseteq A_i|A_j$; hence $\C(\F)$ p-entails $A_i|A_j$;  moreover, $\F$ p-entails $\C(\F)$ and then  $\F$ p-entails $A_i|A_j$.
\begin{remark}
By  Definition \ref{ENTAIL-FAM} and formulas (\ref{QC-LOOP}) and (\ref{LOOP-2}), for any given derangement $(i_1,i_2,\ldots,i_n)$ of $(1,2,\ldots,n)$, we obtain the following inference rule ({\em Generalized Loop})
\begin{equation}\label{GEN-LOOP}
\{A_2|A_1, \ldots, A_{n}|A_{n-1}, A_1|A_n\} \begin{array}{l}\Rightarrow_p\\ \Leftarrow_p\end{array}  \{A_{i_1}|A_1, \ldots, A_{i_{n-1}}|A_{n-1}, A_{i_n}|A_n\}
\end{equation}
\end{remark}
The Loop rule has been studied by a direct probabilistic reasoning in \cite{Gili04}, by exploiting a suitable probabilistic condition named {\em Cs\'asz\'ar's condition}, studied in the framework of an axiomatic approach to probability in \cite{Csas55}. This condition in a particular case reduces to the third axiom of conditional probabilities.
A numerical inference rule named \textit{generalized Bayes theorem}, connected with Cs\'asz\'ar's condition and with Loop rule, has been studied in \cite{AmDP91}; see also \cite{AmDP94,DuPT90}. Below, we reconsider  an example introduced in \cite{Gili04} to  illustrate the generalized Loop rule and  p-entailment of $n$-conditionals.
\begin{example}
Five friends, Linda, Janet, Steve, George, and Peter, have been invited to a party. We define the events:  $A_1=$``\emph{Linda goes to the  party}', \ldots, $A_5=$\emph{Peter goes to the party}; moreover, we assume that $A_1, \ldots, A_5$ are logically independent.
We consider the following  knowledge base:
$\{$``\emph{if Linda goes to the party, then Janet will do the same}'',
\ldots,
``\emph{if George goes to the party, then  Peter will do the same}'',
``\emph{if Peter goes to the party, then  Linda will do the same}''$\}$. Then, for the associated  (p-consistent) family of conditional events $\F=\{A_2|A_1, \ldots, A_{5}|A_4, A_1|A_5\}$, we have
\[
\C(\F)=A_1A_2\cdots A_5|(A_1\vee A_2\vee \cdots\vee A_5)=A_1\dashv\vdash A_2\dashv\vdash \cdots \dashv\vdash A_5\,.
\]
By generalized Loop rule, for every derangement $(i_1,\ldots,i_5)$ of $(1,\ldots,5)$, it holds that
\[
\{A_2|A_1, \ldots, A_{5}|A_{4}, A_1|A_5\} \begin{array}{l}\Rightarrow_p\\ \Leftarrow_p\end{array}  \{A_{i_1}|A_1, \ldots, A_{i_{4}}|A_{4}, A_{i_5}|A_5\}\,.
\]
For any given subset  $\{B_1,\ldots,B_n\}\subset \{A_1,\ldots,A_5\}$,  $n=2,3,4$,  we have  $B_1\dashv\vdash \cdots \dashv\vdash B_n=B_1\cdots B_n|(B_1\vee \cdots\vee B_n)$. This  $n$-conditional is  associated with the conditional  assertion ``\emph{if at least one of $n$ given friends  among Linda, Janet, Steve, George, and Peter, goes to the party, then  all $n$ friends will go to the party}''.
We have
\[
A_1\cdots A_5|(A_1\vee \cdots\vee A_5) \subseteq B_1\cdots B_n|(B_1\vee \cdots\vee B_n)\,;
\]
therefore  $A_1\dashv\vdash \cdots \dashv\vdash A_5$ p-entails $B_1\dashv\vdash \cdots \dashv\vdash B_n$. Finally, as $\F$ p-entails $\C(\F)$, we have that for every subset  $\{B_1,\ldots,B_n\}$, $n=2,3,4$,
the family $\F$ p-entails the $n$-conditional  $B_1\dashv\vdash \cdots \dashv\vdash B_n$.
\end{example}
\section{Conclusions}
In this paper we have examined probabilistic concepts connected with
the inference rules Quasi And, Quasi Or, Or, and generalized Loop.
These are linked with Adams' probabilistic analysis of conditionals,
and play an important role in applications to nonmonotonic
reasoning, to the psychology of uncertain reasoning and to semantic
web. We have considered, in a coherence-based setting, the
extensions of a given probability assessment on $n$ conditional
events to their quasi conjunction and quasi disjunction, by also
examining some cases of logical dependencies. In our probabilistic
analysis we have shown that the lower and upper probability bounds
computed in the different cases coincide with some well known
t-norms and t-conorms: minimum, product, Lukasiewicz and Hamacher
t-norms, and their dual t-conorms. We have shown that, for the Or
rule, the quasi conjunction and quasi disjunction of the premises
are equal.  Moreover, they coincide with the conclusion of the rule.
We have identified the relationships among coherence, inclusion
relation and p-entailment. Finally, we have considered biconditional
events and we have introduced the notion of $n$-conditional event,
by obtaining a probabilistic interpretation for a generalized Loop
rule. In Appendix C we give explicit expressions for the Hamacher
t-norm and t-conorm in the unitary hypercube  $[0,1]^k$. As a "take
home message", the results obtained in our coherence-based
probabilistic approach can be exploited in all researches in
nonmonotonic reasoning, as made for instance in
\cite{GiOv12,Klei13,PfKl06,PfKl09}. Future work should deepen the
theoretical aspects and applications which connect conditional
probability with t-norms and t-conorms, in  relation to inference
patterns in nonmonotonic reasoning. In particular, the
representation of probability bounds for the conditional conclusions
of some inference patterns involving conditionals in terms of
t-norms and t-conorms is a topic that could be expanded. Finally, a
relevant topic for further research concerns the study of more
general definitions for the logical operations of conjunction and
disjunction among conditionals. Such new logical operations should
be defined in a way such that the usual probabilistic properties be
preserved. Some results on this topic have been given in
\cite{GiSa13a}.
\\ \ \\
{\bf Acknowledgements.}
The authors  thank the editors and  four anonymous reviewers for their valuable criticisms and comments, which were helpful in improving the paper.
The authors also acknowledge Frank Lad for his useful suggestions regarding the linguistic quality  of some parts of the manuscript.

\newpage
\appendix
\section{t-norms and t-conorms.}
\label{T-NORM}
We recall below the notions of t-norm and t-conorm  (see \cite{GMMP11,KlMP00,KlMP05}).
\begin{definition}{\rm
A {\em t-norm} is a function $T:[0,1]^2\longrightarrow [0,1]$ which satisfies,
for all $x, y, z \in [0,1]$, the following four axioms:
\[
\begin{array}{lll}
(T1) \hspace{2 mm} & T(x,y) = T(y,x), & \hspace{5 mm} (commutativity) \\
(T2) \hspace{2 mm} & T(x,T(y,z)) = T(T(x,y),z), & \hspace{5 mm} (associativity) \\
(T3) \hspace{2 mm} & T(x,y) \leq T(x,z) \;\; \mbox{whenever} \; y \leq z, & \hspace{5 mm} (monotonicity) \\
(T4) \hspace{2 mm} & T(x,1) = x. & \hspace{5 mm} (boundary \;
condition)
\end{array}
\]
}\end{definition}
We recall below some basic t-norms, namely, the \emph{minimum}  $T_M$ (which is the greatest t-norm), the product $T_P$, the \emph{{\L}ukasiewicz} t-norm $T_L$:
\[
\begin{array}{lll}
T_M(x,y) = \mbox{min}(x,y), & T_P(x,y) = x\cdot y\, , & T_L(x,y) = \mbox{max}(x+y-1, 0).
\end{array}
\]
We also recall that   the \emph{Hamacher} t-norm $T_{\lambda}^H$, with parameter $\lambda \in [0,
\infty]$, is
\begin{equation}\label{HAM-T}
T_{\lambda}^H(x,y)
= \left\{\begin{array}{ll} T_D(x,y), & \lambda=\infty, \\
0, & \lambda=0 \; \mbox{and} \; (x,y)=(0,0) \\
\;\frac{xy}{\lambda+(1-\lambda)(x+y-xy)}, &
\mbox{otherwise},
\end{array}\right.
\end{equation}
where the t-norm $T_D(x,y)$ ({\em drastic product}) is defined as
\[
T_D(x,y) = \left\{\begin{array}{ll} 0, & (x,y) \in [0,1)^2, \\
\mbox{min}(x,y), & \mbox{otherwise}\,.
\end{array}\right.
\]
In particular, the Hamacher t-norm $T_1^H$ is the product t-norm $T_p$.
\begin{definition}{\rm
A {\em t-conorm} is a function $S: [0,1]^2 \longrightarrow [0,1]$
which satisfies, for all $x, y, z \in [0,1]$,  $(T1)-(T3)$ and
\[
\begin{array}{lll}
(S4) \hspace{5 mm} & S(x,0) = x. & \hspace{15 mm} (boundary \;
condition)
\end{array}
\]
}\end{definition}
T-conorms can  be equivalently introduced as dual operations of t-norms. A function $S: [0,1]^2 \longrightarrow [0,1]$, is a t-conorm if and only if there exists a t-norm $T$ such that for all $(x,y)\in [0,1]^2$ either one of the two equalities holds:
$S(x,y)=1-T(1-x,1-y)\,,   T(x,y)=1-S(1-x,1-y)\,.$ Then, the dual t-conorm of $T_M$ is the \emph{maximum} $S_M$, i.e. $S_M(x,y) = \mbox{max}(x,y)$.
The dual t-conorm of $T_P$ is the \emph{probabilistic sum} $S_P$, i.e.
\[
S_P(x,y) = 1-(1-x)(1-y)=x+y-x\cdot y.
\]
The {\L}ukasiewicz t-conorm, which is the dual t-conorm of $T_L$, is
\[
S_L(x,y)=\min(x+y,1)\,.
\]
Moreover, the Hamacher t-conorm $S_{\lambda}^H$  with parameter $\lambda \in [0,\infty]$,  which is the dual t-conorm of $T_{\lambda}^H$, is
\begin{equation}\label{HAM-S}
S_{\lambda}^H(x,y)
= \left\{\begin{array}{ll} S_D(x,y), & \lambda=\infty, \\
1, & \lambda=0 \; \mbox{and} \; x=y=1, \\
\;\frac{x+y-xy-(1-\lambda)xy}{1-(1-\lambda)xy}, &
\mbox{otherwise},
\end{array}\right.
\end{equation}
where  the t-conorm $S_D(x,y)$ ({\em drastic sum}) is defined as
\[
S_D(x,y) = \left\{\begin{array}{ll} 1, & (x,y) \in (0,1]^2, \\
\mbox{max}(x,y), & \mbox{otherwise}.
\end{array}\right.
\]
In particular, the Hamacher t-conorm $S_1^H$ is the probabilistic sum $S_p$.
\section{t-norms and t-conorms in $[0,1]^k$.}
\label{SEZ-TNORMN}
We recall that  since  t-norms and t-conorms are associative they can be easily extended in a unique way to a $k$-ary operation for arbitrary integer $k\geq 2 $  by induction (see \cite{GMMP09,GMMP11,KlMP05}).
Let $T$ be a t-norm (introduced as a binary operator), for any integer $k\geq 2$ the extension of $T$  is defined
as
 \[
T(p_1,p_2,\ldots,p_k)=\left\{ \begin{array}{ll}
 T(T(p_1,\ldots,p_{k-1}),p_k), & $if $ k>  2,\\
 T(p_1,p_2), & $if $ k=  2.
 \end{array}\right.
\]
Let $S$ be a t-conorm (introduced as a binary operator), for any integer $k\in\mathbb{N}\bigcup \{0\} $ the extension of $S$  is defined
as
\[
S(p_1,p_2,\ldots,p_k)=\left\{ \begin{array}{ll}
 S(S(p_1,\ldots,p_{k-1}),p_k), & $if $ k>2,\\
 S(p_1,p_2), & $if $ k=2.
 \end{array}\right.
\]
If $(T,S)$ is a pair of mutually dual t-norms and t-conorms, then
\begin{eqnarray*}
\label{TCONORMN}  S(p_1,\ldots,p_k)=1-T(1-p_1,\ldots,1-p_k)\,,\\
\label{TNORMN} T(p_1,\ldots,p_k)=1-S(1-p_1,\ldots,1-p_k)\,.
\end{eqnarray*}
Finally, we recall that
\[
\begin{array}{ll}
 T_M(p_1,\ldots,p_k) = \mbox{min}(p_1,\ldots,p_k), \;\; S_M(p_1,\ldots,p_k) = \mbox{max}(p_1,\ldots,p_k)\,,\\
 T_p(p_1,\ldots,p_k) = p_1 \cdots p_k, \;\;\;
 S_p(p_1,\ldots,p_k) = 1-(1-p_1)\cdots (1-p_k),\\
T_L(p_1,p_2,\ldots,p_k) = \max(p_1+p_2+\ldots+p_k-(k-1), 0),\\
 S_L(p_1,p_2,\ldots,p_k) =  \min(p_1+p_2+\ldots+p_k,1).
\end{array}
\]

\section{Hamacher t-norm and t-conorm in $[0,1]^k$} \label{HAM-NORM-CON}
In this appendix, by using the notion of additive generator, we  give  self contained constructions  of the extensions of the Hamacher t-norm  and  t-conorm with $\lambda=0$  to $[0,1]^k$. \\
We recall the notion of an additive generator (if any) of a t-norm (\cite{KlMP00,KlMP05}).
\begin{definition}
An additive generator $t:[0,1]\longrightarrow [0,\infty]$ of a t-norm $T$ is a strictly decreasing function which is also right continuous in 0 and satisfies $t(1)=0$, such that for all $(x,y)\in[0,1]^2$ we have
\[
t(x)+t(y)\in Ran(t)\cup [t(0), \infty]
\;\mbox{ and }\;
T(x,y)=t^{-1}(t(x)+t(y))\;,
\]
where $Ran(t)=\{t(x): x\in[0,1]\}$ and  $t^{-1}$ is the pseudo inverse of $t$.
\end{definition}
If $t$ is an additive generator of some t-norm $T$, then we have
\begin{equation}\label{t-generatorn}
T(p_1,p_2,\ldots,p_k)=t^{-1}(t(p_1)+t(p_2)+\ldots+t(p_k)) \,.
\end{equation}

We first observe that: if   $(x=0, y=0)$, then $T_0^H(x,y)=0$; if   $(x=0, y>0)$ or $(x>0,y=0)$, then  $T_0^H(x,y)= \frac{xy}{x+y-xy}=0$; if  $(x>0,y>0)$, then
\[
\begin{array}{ll}\label{EQ-T0<1}
T_0^H(x,y)=\frac{xy}{x+y-xy}=\frac{xy}{x(1-y)+y(1-x)+xy}
=\frac{1}{
\frac{1-x}{x}+\frac{1-y}{y}+1}
>0 .
\end{array}
\]Thus, $T_0^H$  can be equivalently  redefined  as
\begin{equation}\label{HAM-BIS}
 T_0^H(x,y) =
\left\{\begin{array}{ll} 0, & (x=0)\vee (y=0)\,,
\\ \frac{1}{\frac{1-x}{x}+\frac{1-y}{y}+1}, & (x\neq 0)\wedge (y\neq 0)\,.
\end{array}\right.
\end{equation}
We have (see also \cite{GMMP09,KlMP00})
\begin{proposition}\label{PROP-T0N}
Let \(T_{0}^{H}\) be the Hamacher t-norm with  $\lambda=0  $. Given an integer $k\geq 2$, the extension of $T_0^H$ to $[0,1]^k$ is
\begin{equation}\label{EQ-T0N}
T_0^H(p_1,p_2,\ldots,p_k)
= \left\{\begin{array}{ll} 0, &  p_i=0 \mbox{ for at least one } i,
\\
\frac{1
}
{
\sum_{i=1}^k\frac{1-p_i}{p_i}+1
}
 \,, & p_i> 0 \mbox{ for } i=1,\ldots, k\,.
\end{array}\right.
\end{equation}
\end{proposition}
\begin{proof}
We observe that, considering the function $t:[0,1]\longrightarrow [0,+\infty]$ defined as $t(x)=\frac{1-x}{x}$, with the convention that $t(0)=\lim_{x\rightarrow 0^+}\frac{1-x}{x}=+\infty$, it holds
$t^{-1}(s)=\frac{1}{1+s}$, if $s \in [0,\infty]$, with $t^{-1}(+\infty)=0$. Then,  by applying the conventions $\frac{1}{\infty}=0, \frac{1}{0}=\infty$  and recalling (\ref{HAM-BIS}),   for every $(x,y)\in [0,1]^2$ we have
\[
t^{-1}(t(x) + t(y))=\frac{1}{1+
\frac{1-x}{x}+\frac{1-y}{y}}=T_0^H(x,y) ;
\]
As the function $t(x)=\frac{1-x}{x}$ is the additive generator of $T_0^H$, we have
\[
\begin{array}{l}
T_0^H(p_1,p_2,\ldots,p_k) = t^{-1}(\sum_{i=1}^kt(p_i))
= t^{-1}(\sum_{i=1}^k\frac{1-p_i}{p_i})=
\frac{1}{1 + \sum_{i=1}^k\frac{1-p_i}{p_i}} \,.
\end{array}
\]
\end{proof}

Now, we observe that: if $x=1$ and $y=1$, then $S_0^H(x,y)=1$;
if $(x=1, y<1)$ or $(x<1,y=1)$, then $S_0^H(x,y)=\frac{x+y-2xy}{1-xy}=1$;
 if $x<1$ and $y<1$ we have
\[
\begin{array}{ll}\label{EQ-S0<1}S_0^H(x,y)=
\frac{x+y-2xy}{1-xy}=\frac{x(1-y)+y(1-x)}{x(1-y)+y(1-x)+(1-x)(1-y)}
=\\ \\ \frac{\frac{x}{(1-x)}(1-x)(1-y)+\frac{y}{(1-y)}(1-x)(1-y)}{
\frac{x}{(1-x)}(1-x)(1-y)+\frac{y}{(1-y)}(1-x)(1-y)+(1-x)(1-y)}=
\frac{\frac{x}{(1-x)}+\frac{y}{(1-y)}}{\frac{x}{(1-x)}+\frac{y}{(1-y)}+1}<1\,.
\end{array}
\]

Thus, the Hamacher  t-conorm $S_0^H: [0,1]^2 \longrightarrow [0,1]$  can be equivalently  redefined  as
\begin{equation}\label{EQ-S0-BIS}
 S_0^H(x,y) = \left\{\begin{array}{ll} 1, & (x=1)\vee(y=1),
\\ \frac{\frac{x}{(1-x)}+\frac{y}{(1-y)}}{\frac{x}{(1-x)}+\frac{y}{(1-y)}+1}, &
(x< 1)\wedge ( y< 1).
\end{array}\right.
\end{equation}
By observing that $S(p_1,p_2,\ldots,p_k)=1-T(1-p_1,1-p_2,\ldots,1-p_k)$,
it immediately follows
\begin{proposition}\label{PROP-S0N}
Let \(S_{0}^{H}\) be the Hamacher t-conorm with  $\lambda=0  $. Given an integer $k\geq 2$, for any vector $(p_1,p_2,\ldots,p_k) \in[0,1]^k$ it holds that
\begin{equation}\label{EQ-S0N}
S_0^H(p_1,p_2,\ldots,p_k)
= \left\{\begin{array}{ll} 1, &  p_i=1 \mbox{ for at least one } i,
\\
\frac{\sum_{i=1}^k\frac{p_i}{1-p_i}
}
{
\sum_{i=1}^k\frac{p_i}{1-p_i}+1
}
 \,, & p_i< 1 \mbox{ for } i=1,\ldots, k.
\end{array}\right.
\end{equation}
\end{proposition}
\begin{remark}
We observe that the Hamacher t-norm $T_0^H$ and Hamacher \linebreak t-conorm $S_0^H$ coincide, respectively for $\alpha = 1$ and $\alpha = -1$, with the Dombi operator defined as (\cite{Domb09}):
\[
o(p_1,\ldots,p_k) = \frac{1}{1 + \left(\sum_{i=1}^k\left(\frac{1-p_i}{p_i}\right)^{\alpha}\right)^{\frac{1}{\alpha}}} \,.
\]
\end{remark}

\begin{thebibliography}{66}
\expandafter\ifx\csname natexlab\endcsname\relax\def\natexlab#1{#1}\fi
\providecommand{\bibinfo}[2]{#2}
\ifx\xfnm\relax \def\xfnm[#1]{\unskip,\space#1}\fi
\bibitem[{Adams(1975)}]{Adam75}
\bibinfo{author}{E.W. Adams}, \bibinfo{title}{{The Logic of Conditionals}},
  \bibinfo{publisher}{Reidel}, \bibinfo{address}{Dordrecht},
  \bibinfo{year}{1975}.
\bibitem[{Ali et~al.(1978)Ali, Mikhail and Haq}]{AlMH78}
\bibinfo{author}{M.M. Ali}, \bibinfo{author}{N.N. Mikhail},
  \bibinfo{author}{M.~Haq}, \bibinfo{title}{{A class of bivariate distributions
  including the bivariate logistic}}, \bibinfo{journal}{J. Multivariate Anal.}
  \bibinfo{volume}{8} (\bibinfo{year}{1978}) \bibinfo{pages}{405--412}.
\bibitem[{Alsina et~al.(2006)Alsina, Frank and Schweizer}]{AlFS06}
\bibinfo{author}{C.~Alsina}, \bibinfo{author}{M.J. Frank},
  \bibinfo{author}{B.~Schweizer}, \bibinfo{title}{{Associative Functions:
  Triangular Norms and Copulas}}, \bibinfo{publisher}{World Scientific},
  \bibinfo{year}{2006}.
\bibitem[{Amarger et~al.(1991)Amarger, Dubois and Prade}]{AmDP91}
\bibinfo{author}{S.~Amarger}, \bibinfo{author}{D.~Dubois},
  \bibinfo{author}{H.~Prade}, \bibinfo{title}{{Constraint Propagation with
  Imprecise Conditional Probabilities}}, in: \bibinfo{booktitle}{Proc. of the
  7th Conf. on Uncertainty in Artificial Intelligence (UAI-91)},
  \bibinfo{publisher}{Morgan Kaufmann}, \bibinfo{year}{1991}, pp.
  \bibinfo{pages}{26--34}.
\bibitem[{Amarger et~al.(1994)Amarger, Dubois and Prade}]{AmDP94}
\bibinfo{author}{S.~Amarger}, \bibinfo{author}{D.~Dubois},
  \bibinfo{author}{H.~Prade}, \bibinfo{title}{{Handling imprecisely-known
  conditional probabilities}}, in: \bibinfo{editor}{D.J. Hand} (Ed.),
  \bibinfo{booktitle}{AI and Computer Power: The Impact on Statistics},
  \bibinfo{publisher}{Chapman \& Hall}, \bibinfo{year}{1994}, pp.
  \bibinfo{pages}{63--97}.
\bibitem[{Benferhat et~al.(1997)Benferhat, Dubois and Prade}]{BeDP97}
\bibinfo{author}{S.~Benferhat}, \bibinfo{author}{D.~Dubois},
  \bibinfo{author}{H.~Prade}, \bibinfo{title}{{Nonmonotonic Reasoning,
  Conditional Objects and Possibility Theory}}, \bibinfo{journal}{Artif.
  Intell.} \bibinfo{volume}{92} (\bibinfo{year}{1997})
  \bibinfo{pages}{259--276}.
\bibitem[{Biazzo and Gilio(2000)}]{BiGi00}
\bibinfo{author}{V.~Biazzo}, \bibinfo{author}{A.~Gilio}, \bibinfo{title}{{A
  generalization of the fundamental theorem of de Finetti for imprecise
  conditional probability assessments}}, \bibinfo{journal}{Internat. J. Approx.
  Reason.} \bibinfo{volume}{24} (\bibinfo{year}{2000})
  \bibinfo{pages}{251--272}.
\bibitem[{Biazzo et~al.(2002)Biazzo, Gilio, Lukasiewicz and
  Sanfilippo}]{BGLS02}
\bibinfo{author}{V.~Biazzo}, \bibinfo{author}{A.~Gilio},
  \bibinfo{author}{T.~Lukasiewicz}, \bibinfo{author}{G.~Sanfilippo},
  \bibinfo{title}{{Probabilistic logic under coherence, model-theoretic
  probabilistic logic, and default reasoning in System P }},
  \bibinfo{journal}{J. Appl. Non-Classical Logics} \bibinfo{volume}{12}
  (\bibinfo{year}{2002}) \bibinfo{pages}{189--213}.
\bibitem[{Biazzo et~al.(2005)Biazzo, Gilio, Lukasiewicz and
  Sanfilippo}]{BGLS05}
\bibinfo{author}{V.~Biazzo}, \bibinfo{author}{A.~Gilio},
  \bibinfo{author}{T.~Lukasiewicz}, \bibinfo{author}{G.~Sanfilippo},
  \bibinfo{title}{{Probabilistic logic under coherence: complexity and
  algorithms.}}, \bibinfo{journal}{Ann. Math. Artif. Intell.}
  \bibinfo{volume}{45} (\bibinfo{year}{2005}) \bibinfo{pages}{35--81}.
\bibitem[{Biazzo et~al.(2003{\natexlab{a}})Biazzo, Gilio and
  Sanfilippo}]{BiGS03}
\bibinfo{author}{V.~Biazzo}, \bibinfo{author}{A.~Gilio},
  \bibinfo{author}{G.~Sanfilippo}, \bibinfo{title}{{Coherence checking and
  propagation of lower probability bounds}}, \bibinfo{journal}{Soft Computing}
  \bibinfo{volume}{7} (\bibinfo{year}{2003}{\natexlab{a}})
  \bibinfo{pages}{310--320}.
\bibitem[{Biazzo et~al.(2003{\natexlab{b}})Biazzo, Gilio and
  Sanfilippo}]{BiGS03b}
\bibinfo{author}{V.~Biazzo}, \bibinfo{author}{A.~Gilio},
  \bibinfo{author}{G.~Sanfilippo}, \bibinfo{title}{{On the Checking of
  G-Coherence of Conditional Probability Bounds.}}, \bibinfo{journal}{Internat.
  J. Uncertain. Fuzziness Knowledge-Based Systems} \bibinfo{volume}{11,
  Suppl.2} (\bibinfo{year}{2003}{\natexlab{b}}) \bibinfo{pages}{75--104}.
\bibitem[{Biazzo et~al.(2012)Biazzo, Gilio and Sanfilippo}]{BiGS12}
\bibinfo{author}{V.~Biazzo}, \bibinfo{author}{A.~Gilio},
  \bibinfo{author}{G.~Sanfilippo}, \bibinfo{title}{{Coherent Conditional
  Previsions and Proper Scoring Rules}}, in: \bibinfo{editor}{S.~Greco},
  \bibinfo{editor}{B.~Bouchon-Meunier}, \bibinfo{editor}{G.~Coletti},
  \bibinfo{editor}{M.~Fedrizzi}, \bibinfo{editor}{B.~Matarazzo},
  \bibinfo{editor}{R.R. Yager} (Eds.), \bibinfo{booktitle}{Advances in
  Computational Intelligence}, volume \bibinfo{volume}{300} of
  \textit{\bibinfo{series}{CCIS}}, \bibinfo{publisher}{Springer},
  \bibinfo{year}{2012}, pp. \bibinfo{pages}{146--156}.
\bibitem[{Brozzi et~al.(2012)Brozzi, Capotorti and Vantaggi}]{BrCV12}
\bibinfo{author}{A.~Brozzi}, \bibinfo{author}{A.~Capotorti},
  \bibinfo{author}{B.~Vantaggi}, \bibinfo{title}{{Incoherence correction
  strategies in statistical matching}}, \bibinfo{journal}{Int. J. Approx.
  Reason.} \bibinfo{volume}{53} (\bibinfo{year}{2012})
  \bibinfo{pages}{1124--1136}.
\bibitem[{Capotorti et~al.(2007)Capotorti, Lad and Sanfilippo}]{CaLS07}
\bibinfo{author}{A.~Capotorti}, \bibinfo{author}{F.~Lad},
  \bibinfo{author}{G.~Sanfilippo}, \bibinfo{title}{{Reassessing Accuracy Rates
  of Median Decisions}}, \bibinfo{journal}{The American Statistician}
  \bibinfo{volume}{61} (\bibinfo{year}{2007}) \bibinfo{pages}{132--138}.
\bibitem[{Capotorti et~al.(2010)Capotorti, Regoli and Vattari}]{CaRV10}
\bibinfo{author}{A.~Capotorti}, \bibinfo{author}{G.~Regoli},
  \bibinfo{author}{F.~Vattari}, \bibinfo{title}{{Correction of incoherent
  conditional probability assessments}}, \bibinfo{journal}{Int. J. Approx.
  Reason.} \bibinfo{volume}{51} (\bibinfo{year}{2010})
  \bibinfo{pages}{718--727}.
\bibitem[{Ciucci and Dubois(2012)}]{CiDu12}
\bibinfo{author}{D.~Ciucci}, \bibinfo{author}{D.~Dubois},
  \bibinfo{title}{{Relationships between Connectives in Three-Valued Logics}},
  in: \bibinfo{editor}{S.~Greco}, \bibinfo{editor}{B.~Bouchon-Meunier},
  \bibinfo{editor}{G.~Coletti}, \bibinfo{editor}{M.~Fedrizzi},
  \bibinfo{editor}{B.~Matarazzo}, \bibinfo{editor}{R.~Yager} (Eds.),
  \bibinfo{booktitle}{Advances on Computational Intelligence}, volume
  \bibinfo{volume}{297} of \textit{\bibinfo{series}{CCIS}},
  \bibinfo{publisher}{Springer}, \bibinfo{year}{2012}, pp.
  \bibinfo{pages}{633--642}.
\bibitem[{Coletti et~al.(2012{\natexlab{a}})Coletti, Gervasi, Tasso and
  Vantaggi}]{CGTV12}
\bibinfo{author}{G.~Coletti}, \bibinfo{author}{O.~Gervasi},
  \bibinfo{author}{S.~Tasso}, \bibinfo{author}{B.~Vantaggi},
  \bibinfo{title}{{Generalized Bayesian inference in a fuzzy context: From
  theory to a virtual reality application}}, \bibinfo{journal}{Comput. Statist.
  Data Anal.} \bibinfo{volume}{56} (\bibinfo{year}{2012}{\natexlab{a}})
  \bibinfo{pages}{967--980}.
\bibitem[{Coletti and Scozzafava(2002)}]{CoSc02}
\bibinfo{author}{G.~Coletti}, \bibinfo{author}{R.~Scozzafava},
  \bibinfo{title}{{Probabilistic logic in a coherent setting}},
  volume~\bibinfo{volume}{15} of \textit{\bibinfo{series}{Trends in logics}},
  \bibinfo{publisher}{Kluwer}, \bibinfo{address}{Dordrecht},
  \bibinfo{year}{2002}.
\bibitem[{Coletti et~al.(2012{\natexlab{b}})Coletti, Scozzafava and
  Vantaggi}]{CoSV12}
\bibinfo{author}{G.~Coletti}, \bibinfo{author}{R.~Scozzafava},
  \bibinfo{author}{B.~Vantaggi}, \bibinfo{title}{{Inferential processes leading
  to possibility and necessity}}, \bibinfo{journal}{Information Sciences}
  (\bibinfo{year}{2012}{\natexlab{b}}). \bibinfo{note}{Doi
  10.1016/j.ins.2012.10.034}.
\bibitem[{Cs\'{a}sz\'{a}r(1955)}]{Csas55}
\bibinfo{author}{A.~Cs\'{a}sz\'{a}r}, \bibinfo{title}{{Sur la structure des
  espaces de probabilit{\'{e}} conditionnelle}}, \bibinfo{journal}{Acta
  Mathematica Academiae Scientiarum Hungarica} \bibinfo{volume}{6}
  (\bibinfo{year}{1955}) \bibinfo{pages}{337--361}.
\bibitem[{Dombi(2008)}]{Domb09}
\bibinfo{author}{J.~Dombi}, \bibinfo{title}{{Towards a General Class of
  Operators for Fuzzy Systems}}, \bibinfo{journal}{IEEE Trans. Fuzzy Syst.}
  \bibinfo{volume}{16} (\bibinfo{year}{2008}) \bibinfo{pages}{477--484}.
\bibitem[{Dubois and Prade(1994)}]{DuPr94}
\bibinfo{author}{D.~Dubois}, \bibinfo{author}{H.~Prade},
  \bibinfo{title}{{Conditional objects as nonmonotonic consequence
  relationships}}, \bibinfo{journal}{IEEE Trans. Syst., Man, Cybern.}
  \bibinfo{volume}{24} (\bibinfo{year}{1994}) \bibinfo{pages}{1724--1740}.
\bibitem[{Dubois et~al.(1990)Dubois, Prade and Toucas}]{DuPT90}
\bibinfo{author}{D.~Dubois}, \bibinfo{author}{H.~Prade}, \bibinfo{author}{J.M.
  Toucas}, \bibinfo{title}{{Inference with imprecise numerical quantifiers}},
  in: \bibinfo{editor}{Z.~Ras}, \bibinfo{editor}{M.~Zemankova} (Eds.),
  \bibinfo{booktitle}{Intelligent Systems: State of the Art and Future
  Directions}, \bibinfo{publisher}{Ellis Horwood Ltd.}, \bibinfo{year}{1990},
  pp. \bibinfo{pages}{57--72}.
\bibitem[{de~Finetti(1962)}]{deFi62}
\bibinfo{author}{B.~de~Finetti}, \bibinfo{title}{{Does it make sense to speak
  of `good probability appraisers'?}}, in: \bibinfo{editor}{I.J. Good} (Ed.),
  \bibinfo{booktitle}{The scientist speculates: an anthology of partly-baked
  ideas}, \bibinfo{publisher}{Heinemann, London}, \bibinfo{year}{1962}, pp.
  \bibinfo{pages}{357--364}.
\bibitem[{de~Finetti(1964)}]{deFi64}
\bibinfo{author}{B.~de~Finetti}, \bibinfo{title}{{Probabilit\`{a} composte e
  teoria delle decisioni}}, \bibinfo{journal}{Rendiconti di Matematica}
  \bibinfo{volume}{23} (\bibinfo{year}{1964}) \bibinfo{pages}{128--134}.
\bibitem[{de~Finetti(1970)}]{deFi70}
\bibinfo{author}{B.~de~Finetti}, \bibinfo{title}{Teoria delle probabilit\`{a}},
  \bibinfo{publisher}{Ed. Einaudi, 2 voll.}, \bibinfo{address}{Torino},
  \bibinfo{year}{1970}.
\bibitem[{Fugard et~al.(2011)Fugard, Pfeifer, Mayerhofer and Kleiter}]{FPMK11}
\bibinfo{author}{A.J.B. Fugard}, \bibinfo{author}{N.~Pfeifer},
  \bibinfo{author}{B.~Mayerhofer}, \bibinfo{author}{G.D. Kleiter},
  \bibinfo{title}{{How people interpret conditionals: Shifts toward the
  conditional event}}, \bibinfo{journal}{J. Exp. Psychol. Learn. Mem. Cogn.}
  \bibinfo{volume}{37} (\bibinfo{year}{2011}) \bibinfo{pages}{635--648}.
\bibitem[{Gale(1960)}]{Gale60}
\bibinfo{author}{D.~Gale}, \bibinfo{title}{{The theory of linear economic
  models}}, \bibinfo{publisher}{McGraw-Hill}, \bibinfo{address}{NY},
  \bibinfo{year}{1960}.
\bibitem[{Gauffroy and Barrouillet(2009)}]{GaBa09}
\bibinfo{author}{C.~Gauffroy}, \bibinfo{author}{P.~Barrouillet},
  \bibinfo{title}{{Heuristic and analytic processes in mental models for
  conditionals: An integrative developmental theory}},
  \bibinfo{journal}{Developmental Review} \bibinfo{volume}{29}
  (\bibinfo{year}{2009}) \bibinfo{pages}{249--282}.
\bibitem[{Gilio(1990)}]{Gili90}
\bibinfo{author}{A.~Gilio}, \bibinfo{title}{{Criterio di penalizzazione e
  condizioni di coerenza nella valutazione soggettiva della probabilit\`{a}}},
  \bibinfo{journal}{Boll. Un. Mat. Ital.} \bibinfo{volume}{4-B}
  (\bibinfo{year}{1990}) \bibinfo{pages}{645--660}.
\bibitem[{Gilio(1992)}]{Gili92}
\bibinfo{author}{A.~Gilio}, \bibinfo{title}{{$C_0$}-{C}oherence and {E}xtension
  of {C}onditional {P}robabilities}, in: \bibinfo{editor}{J.M. Bernardo},
  \bibinfo{editor}{J.O. Berger}, \bibinfo{editor}{A.P. Dawid},
  \bibinfo{editor}{A.F.M. Smith} (Eds.), \bibinfo{booktitle}{Bayesian
  Statistics 4}, \bibinfo{publisher}{Oxford University Press},
  \bibinfo{year}{1992}, pp. \bibinfo{pages}{633--640}.
\bibitem[{Gilio(1993)}]{Gili93}
\bibinfo{author}{A.~Gilio}, \bibinfo{title}{{Probabilistic Consistency of
  Knowledge Bases in Inference Systems}}, in: \bibinfo{editor}{M.~Clarke},
  \bibinfo{editor}{R.~Kruse}, \bibinfo{editor}{S.~Moral} (Eds.),
  \bibinfo{booktitle}{ECSQARU}, volume \bibinfo{volume}{747} of
  \textit{\bibinfo{series}{LNCS}}, \bibinfo{publisher}{Springer},
  \bibinfo{year}{1993}, pp. \bibinfo{pages}{160--167}.
\bibitem[{Gilio(1995)}]{Gili95}
\bibinfo{author}{A.~Gilio}, \bibinfo{title}{{Algorithms for precise and
  imprecise conditional probability assessments}}, in:
  \bibinfo{editor}{G.~Coletti}, \bibinfo{editor}{D.~Dubois},
  \bibinfo{editor}{R.~Scozzafava} (Eds.), \bibinfo{booktitle}{Mathematical
  Models for Handling Partial Knowledge in Artificial Intelligence},
  \bibinfo{publisher}{Plenum Press}, \bibinfo{address}{New York},
  \bibinfo{year}{1995}, pp. \bibinfo{pages}{231--254}.
\bibitem[{Gilio(1996)}]{Gili96}
\bibinfo{author}{A.~Gilio}, \bibinfo{title}{{Algorithms for conditional
  probability assessments}}, in: \bibinfo{editor}{D.A. Berry},
  \bibinfo{editor}{K.M. Chaloner}, \bibinfo{editor}{J.K. Geweke} (Eds.),
  \bibinfo{booktitle}{Bayesian Analysis in Statistics and Econometrics: Essays
  in Honor of Arnold Zellner}, \bibinfo{publisher}{John Wiley},
  \bibinfo{address}{NY}, \bibinfo{year}{1996}, pp. \bibinfo{pages}{29--39}.
\bibitem[{Gilio(2002)}]{Gili02}
\bibinfo{author}{A.~Gilio}, \bibinfo{title}{{Probabilistic {R}easoning {U}nder
  {C}oherence in {S}ystem {P}}}, \bibinfo{journal}{Ann. Math. Artif. Intell.}
  \bibinfo{volume}{34} (\bibinfo{year}{2002}) \bibinfo{pages}{5--34}.
\bibitem[{Gilio(2004)}]{Gili04}
\bibinfo{author}{A.~Gilio}, \bibinfo{title}{On {Cs\'asz\'ar's Condition in
  Nonmonotonic Reasoning}}, in: \bibinfo{booktitle}{$10$th International
  Workshop on Non-Monotonic Reasoning. Special Session: Uncertainty Frameworks
  in Non-Monotonic Reasoning}, \bibinfo{publisher}{Whistler BC, Canada},
  \bibinfo{address}{June 6--8}, \bibinfo{year}{2004}.
  \bibinfo{note}{\url{http://events.pims.math.ca/science/2004/NMR/uf.html}}.
\bibitem[{Gilio(2012)}]{Gili12}
\bibinfo{author}{A.~Gilio}, \bibinfo{title}{{Generalizing inference rules in a
  coherence-based probabilistic default reasoning}}, \bibinfo{journal}{Int. J.
  Approx. Reasoning} \bibinfo{volume}{53} (\bibinfo{year}{2012})
  \bibinfo{pages}{413--434}.
\bibitem[{Gilio and Over(2012)}]{GiOv12}
\bibinfo{author}{A.~Gilio}, \bibinfo{author}{D.~Over}, \bibinfo{title}{{The
  psychology of inferring conditionals from disjunctions: {A} probabilistic
  study}}, \bibinfo{journal}{Journal of Mathematical Psychology}
  \bibinfo{volume}{56} (\bibinfo{year}{2012}) \bibinfo{pages}{118--131}.
\bibitem[{Gilio and Sanfilippo(2010)}]{GiSa10}
\bibinfo{author}{A.~Gilio}, \bibinfo{author}{G.~Sanfilippo},
  \bibinfo{title}{Quasi {C}onjunction and p-entailment in {N}onmonotonic
  {R}easoning}, in: \bibinfo{editor}{C.~Borgelt},
  \bibinfo{editor}{G.~Rodr{\'i}guez}, \bibinfo{editor}{W.~Trutschnig},
  \bibinfo{editor}{M.A. Lubiano}, \bibinfo{editor}{M.{\'A}. Gil},
  \bibinfo{editor}{P.~Grzegorzewski}, \bibinfo{editor}{O.~Hryniewicz} (Eds.),
  \bibinfo{booktitle}{Combining Soft Computing and Statistical Methods in Data
  Analysis}, volume~\bibinfo{volume}{77} of \textit{\bibinfo{series}{AISC}},
  \bibinfo{publisher}{Springer}, \bibinfo{year}{2010}, pp.
  \bibinfo{pages}{321--328}.
\bibitem[{Gilio and Sanfilippo(2011{\natexlab{a}})}]{GiSa11a}
\bibinfo{author}{A.~Gilio}, \bibinfo{author}{G.~Sanfilippo},
  \bibinfo{title}{{Coherent conditional probabilities and proper scoring
  rules}}, in: \bibinfo{editor}{F.~Coolen}, \bibinfo{editor}{G.~de~Cooman},
  \bibinfo{editor}{T.~Fetz}, \bibinfo{editor}{M.~Oberguggenberger} (Eds.),
  \bibinfo{booktitle}{ISIPTA'11: Proceedings of the Seventh International
  Symposium on Imprecise Probability: Theories and Applications},
  \bibinfo{publisher}{SIPTA}, \bibinfo{address}{Innsbruck},
  \bibinfo{year}{2011}{\natexlab{a}}, pp. \bibinfo{pages}{189--198}.
\bibitem[{Gilio and Sanfilippo(2011{\natexlab{b}})}]{GiSa11b}
\bibinfo{author}{A.~Gilio}, \bibinfo{author}{G.~Sanfilippo},
  \bibinfo{title}{Quasi conjunction and inclusion relation in probabilistic
  default reasoning}, in: \bibinfo{editor}{W.~Liu} (Ed.),
  \bibinfo{booktitle}{Symbolic and Quantitative Approaches to Reasoning with
  Uncertainty}, volume \bibinfo{volume}{6717} of
  \textit{\bibinfo{series}{LNCS}}, \bibinfo{publisher}{Springer},
  \bibinfo{year}{2011}{\natexlab{b}}, pp. \bibinfo{pages}{497--508}.

\bibitem[{Gilio and Sanfilippo(2012)}]{GiSa12}
\bibinfo{author}{A.~Gilio}, \bibinfo{author}{G.~Sanfilippo},
  \bibinfo{title}{{Probabilistic entailment in the setting of coherence: The
  role of quasi conjunction and inclusion relation}}, 
  \bibinfo{journal}{Int. J.  Approx. Reasoning}  \bibinfo{volume}{54} (\bibinfo{year}{2013})  \bibinfo{pages}{513--525}.

\bibitem[{Gilio and Sanfilippo(2013)}]{GiSa13a}
\bibinfo{author}{A.~Gilio}, \bibinfo{author}{G.~Sanfilippo},
  \bibinfo{title}{{C}onjunction, {D}isjunction and {I}terated {C}onditioning of
  {C}onditional {E}vents}, in: \bibinfo{editor}{R.~Kruse},
  \bibinfo{editor}{M.R. Berthold}, \bibinfo{editor}{C.~Moewes},
  \bibinfo{editor}{M.A. Gil}, \bibinfo{editor}{P.~Grzegorzewski},
  \bibinfo{editor}{O.~Hryniewicz} (Eds.), \bibinfo{booktitle}{Synergies of Soft
  Computing and Statistics for Intelligent Data Analysis}, volume
  \bibinfo{volume}{190} of \textit{\bibinfo{series}{Advances in Intelligent
  Systems and Computing}}, \bibinfo{publisher}{Springer}, \bibinfo{year}{2013},
  pp. \bibinfo{pages}{399--407}.
\bibitem[{Godo and Marchioni(2006)}]{GoMa06}
\bibinfo{author}{L.~Godo}, \bibinfo{author}{E.~Marchioni},
  \bibinfo{title}{{Coherent Conditional Probability in a Fuzzy Logic Setting}},
  \bibinfo{journal}{Logic Journal of the IGPL} \bibinfo{volume}{14}
  (\bibinfo{year}{2006}) \bibinfo{pages}{457--481}.
\bibitem[{Goodman and Nguyen(1988)}]{GoNg88}
\bibinfo{author}{I.R. Goodman}, \bibinfo{author}{H.T. Nguyen},
  \bibinfo{title}{{Conditional Objects and the Modeling of Uncertainties}}, in:
  \bibinfo{editor}{M.M. Gupta}, \bibinfo{editor}{T.~Yamakawa} (Eds.),
  \bibinfo{booktitle}{Fuzzy Computing}, \bibinfo{publisher}{North-Holland},
  \bibinfo{year}{1988}, pp. \bibinfo{pages}{119--138}.
\bibitem[{Grabisch et~al.(2009)Grabisch, Marichal, Mesiar and Pap}]{GMMP09}
\bibinfo{author}{M.~Grabisch}, \bibinfo{author}{J.L. Marichal},
  \bibinfo{author}{R.~Mesiar}, \bibinfo{author}{E.~Pap},
  \bibinfo{title}{{Aggregation functions}}, \bibinfo{publisher}{Cambridge
  University Press}, \bibinfo{year}{2009}.
\bibitem[{Grabisch et~al.(2011{\natexlab{a}})Grabisch, Marichal, Mesiar and
  Pap}]{GMMP11b}
\bibinfo{author}{M.~Grabisch}, \bibinfo{author}{J.L. Marichal},
  \bibinfo{author}{R.~Mesiar}, \bibinfo{author}{E.~Pap},
  \bibinfo{title}{{Aggregation functions: Construction methods, conjunctive,
  disjunctive and mixed classes}}, \bibinfo{journal}{Information Sciences}
  \bibinfo{volume}{181} (\bibinfo{year}{2011}{\natexlab{a}})
  \bibinfo{pages}{23--43}.
\bibitem[{Grabisch et~al.(2011{\natexlab{b}})Grabisch, Marichal, Mesiar and
  Pap}]{GMMP11}
\bibinfo{author}{M.~Grabisch}, \bibinfo{author}{J.L. Marichal},
  \bibinfo{author}{R.~Mesiar}, \bibinfo{author}{E.~Pap},
  \bibinfo{title}{{Aggregation functions: Means}},
  \bibinfo{journal}{Information Sciences} \bibinfo{volume}{181}
  (\bibinfo{year}{2011}{\natexlab{b}}) \bibinfo{pages}{1--22}.
\bibitem[{Hamacher(1978)}]{Hama78}
\bibinfo{author}{H.~Hamacher}, \bibinfo{title}{{\"{U}ber logische Aggregationen
  nicht-bin\"{a}r explizierter Entscheidungskriterien}},
  \bibinfo{publisher}{Rita G. Fischer Verlag}, \bibinfo{year}{1978}.
\bibitem[{Kern-Isberner(2001)}]{Kern01}
\bibinfo{author}{G.~Kern-Isberner}, \bibinfo{title}{{Conditionals in
  Nonmonotonic Reasoning and Belief Revision}}, volume \bibinfo{volume}{2087}
  of \textit{\bibinfo{series}{LNCS}}, \bibinfo{publisher}{Springer},
  \bibinfo{year}{2001}.
\bibitem[{Kleiter(2013)}]{Klei13}
\bibinfo{author}{G.~Kleiter}, \bibinfo{title}{{Ockham's Razor in Probability
  Logic}}, in: \bibinfo{editor}{R.~Kruse}, \bibinfo{editor}{M.R. Berthold},
  \bibinfo{editor}{C.~Moewes}, \bibinfo{editor}{M.A. Gil},
  \bibinfo{editor}{P.~Grzegorzewski}, \bibinfo{editor}{O.~Hryniewicz} (Eds.),
  \bibinfo{booktitle}{Synergies of Soft Computing and Statistics for
  Intelligent Data Analysis}, volume \bibinfo{volume}{190} of
  \textit{\bibinfo{series}{Advances in Intelligent Systems and Computing}},
  \bibinfo{publisher}{Springer}, \bibinfo{year}{2013}, pp.
  \bibinfo{pages}{409--417}.
\bibitem[{Klement et~al.(2000)Klement, Mesiar and Pap}]{KlMP00}
\bibinfo{author}{E.P. Klement}, \bibinfo{author}{R.~Mesiar},
  \bibinfo{author}{E.~Pap}, \bibinfo{title}{{Triangular Norms}},
  \bibinfo{publisher}{Springer}, \bibinfo{year}{2000}.
\bibitem[{Klement et~al.(2005)Klement, Mesiar and Pap}]{KlMP05}
\bibinfo{author}{E.P. Klement}, \bibinfo{author}{R.~Mesiar},
  \bibinfo{author}{E.~Pap}, \bibinfo{title}{{Triangular norms: basic notions
  and properties}}, in: \bibinfo{booktitle}{Logical, algebraic, analytic and
  probabilistic aspects of triangular norms}, \bibinfo{publisher}{Elsevier},
  \bibinfo{year}{2005}, pp. \bibinfo{pages}{17--60}.
\bibitem[{Kraus et~al.(1990)Kraus, Lehmann and Magidor}]{KrLM90}
\bibinfo{author}{S.~Kraus}, \bibinfo{author}{D.~Lehmann},
  \bibinfo{author}{M.~Magidor}, \bibinfo{title}{{Nonmonotonic reasoning,
  preferential models and cumulative logics}}, \bibinfo{journal}{Artif.
  Intell.} \bibinfo{volume}{44} (\bibinfo{year}{1990})
  \bibinfo{pages}{167--207}.
\bibitem[{Lad(1996)}]{Lad96}
\bibinfo{author}{F.~Lad}, \bibinfo{title}{{Operational Subjective Statistical
  Methods}}, \bibinfo{publisher}{Wiley}, \bibinfo{year}{1996}.
\bibitem[{Lad et~al.(2012)Lad, Sanfilippo and Agr\'{o}}]{LaSA12}
\bibinfo{author}{F.~Lad}, \bibinfo{author}{G.~Sanfilippo},
  \bibinfo{author}{G.~Agr\'{o}}, \bibinfo{title}{Completing the logarithmic
  scoring rule for assessing probability distributions}, \bibinfo{journal}{AIP
  Conference Proceedings} \bibinfo{volume}{1490} (\bibinfo{year}{2012})
  \bibinfo{pages}{13--30}.
\bibitem[{Lukasiewicz and Straccia(2008)}]{LuSt08}
\bibinfo{author}{T.~Lukasiewicz}, \bibinfo{author}{U.~Straccia},
  \bibinfo{title}{{Managing uncertainty and vagueness in description logics for
  the Semantic Web}}, \bibinfo{journal}{Journal of Web Semantics}
  \bibinfo{volume}{6} (\bibinfo{year}{2008}) \bibinfo{pages}{291--308}.
\bibitem[{Menger(1942)}]{Meng42}
\bibinfo{author}{K.~Menger}, \bibinfo{title}{{Statistical Metrics}},
  \bibinfo{journal}{Proc Natl Acad Sci U S A} \bibinfo{volume}{28}
  (\bibinfo{year}{1942}) \bibinfo{pages}{535--537}.
\bibitem[{Nelsen(1999)}]{Nels99}
\bibinfo{author}{R.B. Nelsen}, \bibinfo{title}{{An Introduction to Copulas}},
  volume \bibinfo{volume}{139} of \textit{\bibinfo{series}{Lecture Notes in
  Statistics}}, \bibinfo{publisher}{Springer}, \bibinfo{year}{1999}.
\bibitem[{Pfeifer and Kleiter(2006)}]{PfKl06}
\bibinfo{author}{N.~Pfeifer}, \bibinfo{author}{G.D. Kleiter},
  \bibinfo{title}{{Inference in conditional probability logic}},
  \bibinfo{journal}{Kybernetika} \bibinfo{volume}{42} (\bibinfo{year}{2006})
  \bibinfo{pages}{391--404}.
\bibitem[{Pfeifer and Kleiter(2009)}]{PfKl09}
\bibinfo{author}{N.~Pfeifer}, \bibinfo{author}{G.D. Kleiter},
  \bibinfo{title}{{Framing human inference by coherence based probability
  logic}}, \bibinfo{journal}{Journal of Applied Logic} \bibinfo{volume}{7}
  (\bibinfo{year}{2009}) \bibinfo{pages}{206--217}.
\bibitem[{Sanfilippo(2012)}]{Sanf12}
\bibinfo{author}{G.~Sanfilippo}, \bibinfo{title}{{From imprecise probability
  assessments to conditional probabilities with quasi additive classes of
  conditioning events}}, in: \bibinfo{booktitle}{Proc. of the Twenty-Eighth
  Conference on Uncertainty in Artificial Intelligence (UAI-12)},
  \bibinfo{publisher}{AUAI Press}, \bibinfo{address}{Corvallis, Oregon},
  \bibinfo{year}{2012}, pp. \bibinfo{pages}{736--745}.
\bibitem[{Schweizer and Sklar(1961)}]{ScSk61}
\bibinfo{author}{B.~Schweizer}, \bibinfo{author}{A.~Sklar},
  \bibinfo{title}{{Associative functions and statistical triangle
  inequalities}}, \bibinfo{journal}{Publ. Math.} \bibinfo{volume}{8}
  (\bibinfo{year}{1961}) \bibinfo{pages}{169--186}.
\bibitem[{Scozzafava and Vantaggi(2009)}]{ScVa09}
\bibinfo{author}{R.~Scozzafava}, \bibinfo{author}{B.~Vantaggi},
  \bibinfo{title}{{Fuzzy inclusion and similarity through coherent conditional
  probability}}, \bibinfo{journal}{Fuzzy Sets and Systems}
  \bibinfo{volume}{160} (\bibinfo{year}{2009}) \bibinfo{pages}{292--305}.
\bibitem[{Thimm et~al.(2011)Thimm, Kern-Isberner and Fisseler}]{TKIF11}
\bibinfo{author}{M.~Thimm}, \bibinfo{author}{G.~Kern-Isberner},
  \bibinfo{author}{J.~Fisseler}, \bibinfo{title}{{Relational Probabilistic
  Conditional Reasoning at Maximum Entropy}}, in: \bibinfo{editor}{W.~Liu}
  (Ed.), \bibinfo{booktitle}{Symbolic and Quantitative Approaches to Reasoning
  with Uncertainty}, volume \bibinfo{volume}{6717} of
  \textit{\bibinfo{series}{LNCS}}, \bibinfo{publisher}{Springer},
  \bibinfo{year}{2011}, pp. \bibinfo{pages}{447--458}.
\bibitem[{Tweney et~al.(2010)Tweney, Doherty and Kleiter}]{TwDK10}
\bibinfo{author}{R.D. Tweney}, \bibinfo{author}{M.E. Doherty},
  \bibinfo{author}{G.D. Kleiter}, \bibinfo{title}{{The pseudodiagnosticity
  trap: Should participants consider alternative hypotheses?}},
  \bibinfo{journal}{Thinking \& Reasoning} \bibinfo{volume}{16}
  (\bibinfo{year}{2010}) \bibinfo{pages}{332--345}.

\end{thebibliography}
\end{document}